\documentclass[12pt]{article}
\usepackage{color,latexsym,amsthm,amsmath,amssymb,path,enumitem,url}

\newcommand{\ignore}[1]{}

\newtheorem{theorem}{Theorem}[section]
\newtheorem{proposition}[theorem]{Proposition}
\newtheorem{lemma}[theorem]{Lemma}
\newtheorem{corollary}[theorem]{Corollary}

\newtheorem*{thmRichC}{Theorem \ref{th:RichC}}
\newtheorem*{thmDDinC2}{Theorem \ref{th:DDinC2}}
\newtheorem*{propositionStructureHypersurfaces}{Proposition \ref{prop:structureHypersurfaces}}

\theoremstyle{remark}

\newtheorem{remark}[theorem]{Remark}

\usepackage{graphicx,psfrag}

\usepackage{dsfont}

\newcommand{\RR}{\ensuremath{\mathbb R}}
\newcommand{\CC}{\ensuremath{\mathbb C}}

\newcommand{\FF}{\ensuremath{\mathbb F}}

\newcommand{\lines}{\mathcal L}
\newcommand{\curves}{\Gamma}
\newcommand{\pts}{\mathcal P}
\newcommand{\surfs}{\mathcal S}
\newcommand{\vrts}{\mathcal V}

\newcommand{\BI}{\mathbf{I}}

\newcommand{\vb}{Z}

\newcommand{\gr}{\mathrm{Gr}}

\newcommand{\sing}{\operatorname{sing}}
\newcommand{\reg}{\operatorname{reg}}

\newcommand{\hair}{H}

\def\eps{{\varepsilon}}

\newcommand{\parag}[1]{\vspace{2mm}

\noindent{\bf #1} }

\begin{document}
\pagenumbering{arabic}

\title{Distinct distances in the complex plane}

\author{
Adam Sheffer\thanks{Department of Mathematics, Baruch College, City University of New York, NY, USA.
{\sl adamsh@gmail.com}.}
\and
Joshua Zahl\thanks{Department of Mathematics, University of British Columbia, Vancouver, BC, Canada {\sl jzahl@math.ubc.ca}.}
}


\maketitle

\begin{abstract}
We prove that if $P$ is a set of $n$ points in $\mathbb{C}^2$, then either the points in $P$ determine $\Omega(n^{1-\eps})$ complex distances, or $P$ is contained in a line with slope $\pm i$. If the latter occurs then each pair of points in $P$ have complex distance 0. 
\end{abstract}

\section{Introduction}
In 1946, Erd\H os \cite{erd46} posed the question: how few distinct distances can be determined by a set of $n$ points in the plane? The Erd\H os distinct distances problem has become a central question in combinatorial geometry, and an entire book has been dedicated to the question \cite{GIS11}. In 2010, Guth and Katz \cite{GK15} nearly resolved the conjecture by establishing the following lower bound.
\begin{theorem} \label{th:GuthKatz}
Every set of $n$ points in $\RR^2$ determines $\Omega(n/\log n)$ distinct distances.
\end{theorem}
This lower bound nearly matches the conjectured lower bound $\Omega(n/\sqrt{\log n})$, which can be achieved by taking points of the form $(j,k)$ with $j$ and $k$ integers between $1$ and $\sqrt n$. Theorem \ref{th:GuthKatz} bookends decades of progress on the Erd\H os distinct distances problem, such as \cite{CST92,SoTo01,Tardos03}. The problem has also been studied in other fields and under different distance norms \cite{BKT04,SdZ17,MPPRS20}. See \cite{Sheffer14} for a survey of recent results.

In this paper we obtain an analogue of Theorem \ref{th:GuthKatz} for sets of points in $\CC^2$. If $p,q\in\CC^2$ we define the \emph{(squared) complex distance} $\Delta(p,q) = (p_x-q_x)^2 + (p_y-q_y)^2$. For $\pts\subset\CC^2$, we define 
$$
\Delta(\pts) = \{ \Delta(p,q) \colon p, q \in\pts,\ p\neq q\}.
$$
In contrast to the situation in $\RR^2$, the set $\Delta(\pts)$ can contain the distance 0. Indeed, it is possible that $\Delta(\pts) =\{0\}$ even when $\pts$ is large. We say that a line $L\subset\CC^2$ is \emph{isotropic} if it has a slope of $\pm i$. If two points $p,q\in\CC^2$ are contained in a common isotropic line, then $\Delta(p,q)=0$. In particular, if all points of $\pts$ are contained in a common isotropic line then $\Delta(P)=\{0\}$. The next theorem says that this is the only obstruction preventing $\Delta(\pts)$ from having large cardinality. 

\begin{theorem}[Distinct distances in $\CC^2$] \label{th:DDinC2}
For every $\eps>0$, there is a positive constant $c$ such that the following holds. Let $\pts$ be a set of $n$ points in $\CC^2$, not all on the same isotropic line. Then 
\begin{equation*}
|\Delta(\mathcal{P})|\geq cn^{1-\eps}.
\end{equation*}
\end{theorem}

Theorem \ref{th:DDinC2} yields several new sum-product type estimates for finite sets of complex numbers. 
\begin{corollary}
For every $\eps>0$, there is a positive constant $c$ such that the following holds. Let $A\subset\CC$. Then 
\begin{equation}\label{eq:SumProductOne}
|\{ (a_1-a_2)^2 \pm (a_3-a_4)^2\colon a_1,a_2,a_3,a_4\in A\}|\geq c|A|^{2-\eps}.
\end{equation}
Similarly,
\begin{equation}\label{eq:SumProductTwo}
|\{ (a_1- a_2)(a_3 - a_4)\colon a_1,a_2,a_3,a_4\in A\}|\geq c|A|^{2-\eps}.
\end{equation}
\end{corollary}
The sum-product estimate \eqref{eq:SumProductOne} follows by applying Theorem \ref{th:DDinC2} to the set $\pts = A\times A$ or $\pts = A\times iA$. The estimate \eqref{eq:SumProductTwo} follows by applying Theorem \ref{th:DDinC2} to the set $\pts = \{(a_1+a_2, ia_1-ia_2)\colon a_1,a_2\in A\}$. When $A\subset\RR$, both of these estimates were previously known; \eqref{eq:SumProductOne} with $+$ sign follows immediately from Theorem \ref{th:GuthKatz}, while the other estimates were proved by Roche-Newton and Rudnev in \cite{RR15}. 

\subsection{From distinct distances to incidence geometry}
To prove Theorem \ref{th:GuthKatz}, Guth and Katz used the so-called Elekes-Sharir-Guth-Katz framework. 
This framework reduces Theorem \ref{th:GuthKatz} to an incidence geometry problem about lines in $\RR^3$. In the result that follows, we say that a point $p$ is $r$-rich with respect to a set of lines $\lines$ if at least $r$ lines from $\lines$ contain $p$. We write $\pts_r(\lines)$ to denote the set of points that are $r$-rich with respect to $\lines$.
\begin{theorem}\label{th:numberOfRRichPoints}
Let $\lines$ be a set of at most $n$ lines in $\RR^3$, and suppose that at most $n^{1/2}$ lines are contained in a common plane or doubly-ruled surface. Then for each $2\leq r\leq n^{1/2}$, 
\begin{equation*}
|\pts_r(\lines)|= O(n^{3/2}r^{-2}).
\end{equation*}
\end{theorem}
When $r=2$, Guth and Katz's proof of Theorem \ref{th:numberOfRRichPoints} is purely algebraic, and it has since been extended to arbitrary fields \cite{GZ15a, Kollar15}. For larger values of $r$, the only known proof of Theorem \ref{th:numberOfRRichPoints} requires topological arguments that are specific to $\RR$. Specifically, Guth and Katz developed a new tool called the polynomial partitioning theorem.
\begin{theorem}[Polynomial partitioning]\label{th:partitionPts}
Let $\pts$ be a set of $m$ points in $\RR^d$ and let $r\ge 1$.
Then there exists a nonzero polynomial $f\in \RR[x_1,\ldots,x_d]$ of degree at most $r$, such that each connected component of $\RR^d\setminus \vb(f)$ contains $O(mr^{-d})$ points of $\pts$.
\end{theorem}

Since its introduction in 2010, Theorem \ref{th:partitionPts} has reshaped the field of incidence geometry and has led to striking progress on problems in discrete geometry, theoretical computer science, and harmonic analysis. See \cite{Guth16} for a partial survey of these developments. Many of the incidence geometry problems in Euclidean space that have been solved using Theorem \ref{th:partitionPts} can also be posed in vector spaces over other fields such as $\CC$ or $\mathbb{F}_p$.
Since we do not have an analogue of Theorem \ref{th:partitionPts} in these settings, many of these problems remain open. Up to an $\eps$ loss in the exponent, the following theorem is a complex analogue of Theorem \ref{th:numberOfRRichPoints}.
\begin{theorem}\label{th:numberOfRRichPointsCplx}
For every $\eps>0$, there is a constant $C$ such that the following holds. Let $\lines$ be a set of at most $n$ lines in $\CC^3$, and suppose that at most $n^{1/2}$ lines are contained in a common plane or doubly-ruled surface. Then for each $2\leq r\leq n^{1/2}$,
\begin{equation*}
|\pts_r(\lines)|\leq C n^{3/2+\eps}r^{-2}.
\end{equation*}
\end{theorem}
When $r=2$, Theorem \ref{th:numberOfRRichPointsCplx} follows from the more general results in \cite{GZ15a, Kollar15}. A key difficulty when proving Theorem \ref{th:numberOfRRichPointsCplx} for larger values of $r$ is that the complex analogue of Theorem \ref{th:partitionPts} is false---if $f\in \CC[x_1,\ldots,x_d]$ is a polynomial with $\vb(f)\subset\CC^d$, then $\CC^d\backslash \vb(f)$ is connected, so $\vb(f)$ does not ``cut'' $\CC^d$ into multiple connected components. In the next section we discuss our strategy for overcoming this problem.

\subsection{A structure theorem for lines in three dimensions}
In \cite{Guth15a}, Guth used Theorem \ref{th:partitionPts} to obtain the following structure theorem about sets of lines in $\RR^3$. 
For a set $\lines$ of lines and a variety $W$, we denote by $\lines_W$ the set of curves of $\lines$ that are contained in $W$.

\begin{theorem} \label{th:RichRGuth}
For every $\eps>0$, there are constants $D$ and $C$ such that the following holds. Let $\lines$ be a set of $n$ lines in $\RR^3$, let $2\leq r\leq 2n^{1/2}$, and let $r^\prime = \lceil 9r/10\rceil$. Then there exists a set $\surfs$ of algebraic surfaces in $\RR^3$ with the following properties.
\begin{itemize}[noitemsep]
\item Every surface $W\in \surfs$ is an irreducible surface of degree at most $D$.
\item Each surface contains at least $n^{1/2+\eps}$ lines of $\lines$.
\item $|\surfs|\le 2n^{1/2-\eps}$.
\item $|\pts_r(\lines)\setminus \bigcup_{W\in \surfs}\pts_{r^\prime}(\lines_W)|\leq C n^{3/2+\eps}r^{-2}$.
\end{itemize}
\end{theorem}
Guth then showed that Theorem \ref{th:RichRGuth} implies a slightly weaker version of Theorem \ref{th:numberOfRRichPoints}, which in turn implies a slightly weaker version of Theorem \ref{th:GuthKatz}.
Specifically, the bound $\Omega(n/\log n)$ is replaced by $\Omega(n^{1-\eps})$. We use a similar strategy to prove Theorem \ref{th:numberOfRRichPointsCplx} and Theorem \ref{th:DDinC2}. In particular, we prove the following complex analogue of Theorem \ref{th:RichRGuth}. 

\begin{theorem} \label{th:RichC}
For every $\eps>0$, there is a constant $C$ such that the following holds. Let $\lines$ be a set of $n$ lines in $\CC^3$, let $2\leq r\leq 2n^{1/2}$, and let $r^\prime = \max(2,  r/3)$. Then there exists a set $\surfs$ of algebraic surfaces in $\CC^3$ with the following properties.
\begin{itemize}[noitemsep]
\item If $r\geq 3$ then every surface in $\surfs$ is a plane. If $r=2$ then every surface in $\surfs$ is irreducible and has degree at most two.
\item Every plane $W\in \surfs$ contains at least $rn^{1/2+\eps}$ lines of $\lines$. 
\item $|\surfs|\le 2n^{1/2-\eps}r^{-1}$.
\item $|\pts_r(\lines)\setminus \bigcup_{W\in \surfs}\pts_{r^\prime}(\lines_W)|\leq C n^{3/2+\eps}r^{-2}$.
\end{itemize}
\end{theorem}
Guth proved Theorem \ref{th:RichRGuth} by induction on $n$ using a divide and conquer approach. Given a set of lines $\lines$, Guth used Theorem \ref{th:partitionPts} to find a polynomial $f$ with the following properties. 
The set $\RR^3\backslash \vb(f)$ is a union of many connected regions, each containing a small fraction of the points from $\pts_r(\lines)$, and most intersecting a small fraction of the lines from $\lines$. He then applied the induction hypothesis (Theorem \ref{th:RichRGuth} with fewer lines) to each of these regions individually. Finally, he combined the sets of algebraic surfaces associated to each region into a single, slightly larger set of algebraic surfaces and thereby closed the induction. 

An important technical difficulty in Guth's proof is that some of the $r$-rich points and some of the lines might be contained in the ``boundary'' $\vb(f)$ of the partition. Luckily, the boundary $\vb(f)$ is itself a variety, and irreducible components of this variety that contain many lines can be added to the set $\surfs$.

We now briefly describe our strategy for proving Theorem \ref{th:RichC}. As noted above, the complex analogue of Theorem \ref{th:partitionPts} is false. This is problematic because Theorem \ref{th:partitionPts} played a critical role in Guth's proof of Theorem \ref{th:RichRGuth}. One strategy for proving incidence geometry problems in complex space is to identify $\CC^d$ with $\RR^{2d}$ and to apply the polynomial partitioning theorem in $\RR^{2d}$. This was the approach used by the authors in \cite{SSZ16} to establish new point-curve incidence results in $\CC^2$, and we use a similar strategy to prove Theorem \ref{th:RichC}. 

To execute this strategy, we identify $\CC^3$ with $\RR^6$, and each complex line becomes a real plane. With a slight abuse of notation, we continue to call these sets complex lines. Some of the steps in Guth's proof of Theorem \ref{th:RichRGuth} can still be used to prove Theorem \ref{th:RichC}: We prove the theorem by induction on $n$ and find a real polynomial $f$ with the following properties. The set $\RR^6\backslash \vb(f)$ is a union of many connected regions, each containing far fewer $r$-rich points and intersecting far fewer complex lines than the original problem. As in the proof of Theorem \ref{th:RichRGuth}, it is possible that many $r$-rich points and many complex lines are contained in the  boundary $\vb(f)$ of the partition. Unfortunately, $\vb(f)$ might \emph{not} be a complex variety, so we are not permitted to add it to $\surfs$. Instead, we study the incidence geometry of points and complex lines contained in a real hypersurface in $\RR^6$. Our main result in this direction is the following incidence theorem, which is proved in Section \ref{sec:cplxGeom}. Before stating our result, we require a definition. We say that a variety $U\subset\RR^6$ is \emph{almost ruled by complex planes} if for each regular point $p\in U_{\operatorname{reg}}$, there is a complex plane $\Pi\subset U$ that contains $p$. If such a plane exists and $L\subset U$ is a complex line incident to $p$, then $L\subset\Pi$. 

\begin{proposition}\label{prop:structureHypersurfaces}
Let $U\subset\RR^6$ be an irreducible variety defined by polynomials of degree at most $D$. Then at least one of the following two statements holds.
\begin{itemize}[noitemsep]
\item $U$ is almost ruled by complex planes. 

\item $U$ contains at most $2D^2(D-1)$ complex planes. If $\lines$ is a set of $n$ complex lines that are contained in $U$ but not contained in any of these planes, then for each $r> D^2$ we have
\begin{equation*}
|U_{\operatorname{reg}}\cap \pts_r(\lines)| = O_D(n^{3/2}r^{-5/2}+nr^{-1}).
\end{equation*}
\end{itemize}
\end{proposition}

Proposition \ref{prop:structureHypersurfaces} allows us to deal with the situation where many points and complex lines are contained in $\vb(f)$; this is the final missing piece in the proof of Theorem \ref{th:RichC}.

As a final remark, we note that our choice of $r^\prime$ in Theorem \ref{th:RichC} is slightly different than the choice used in Theorem \ref{th:RichRGuth}. This is a minor technical issue, and it does not limit the usefulness of Theorem \ref{th:RichC}. For a theorem of this type to be useful, we require that $r^\prime = 2$ when $r=2$, and that $r^\prime$ grows linearly as a function of $r$. 

\subsection{Structure of the paper}
In Section \ref{sec:prelim} we introduce a number of tools from algebraic geometry and real algebraic geometry that appear frequently in our proof. We also introduce the ruled surface theory developed by Guth and the second author in \cite{GZ15a}. Proposition \ref{prop:structureHypersurfaces} is a statement about complex lines and the theory of surfaces ruled by complex lines is quite classical. 
However, some of the intermediate steps in the proof require structure theorems about surfaces ruled by more general types of curves. 
Thus, the full strength of the ruled surface theory developed in \cite{GZ15a} is required.

In Section \ref{sec:KollarBd} we show how existing algebraic techniques can be used to prove a special case of Theorem \ref{th:RichC} when $r$ is small---as discussed above, Theorem \ref{th:RichC} is only novel when $r$ is large.  This is helpful because Proposition \ref{prop:structureHypersurfaces} is only effective when $r$ is large. In Section \ref{sec:GuthStructureTheorem} we discuss how a slight variant of Guth's Theorem \ref{th:RichRGuth} also applies to curves in $\RR^d$. This result will help us understand complex lines that properly intersect the boundary of the partition.

In Section \ref{sec:cplxGeom} we prove Proposition \ref{prop:structureHypersurfaces} and in Section \ref{sec:LineIncC3} we use this proposition to prove Theorem \ref{th:RichC}. Finally, in Section \ref{sec:DDinC3} we combine Theorem \ref{th:RichC} with the Elekes-Sharir-Guth-Katz framework to prove Theorem \ref{th:DDinC2}

\subsection{Thanks}
The authors would like to thank Thomas Bloom, Jozsef Solymosi, and Endre Szab\'o for helpful discussions. Adam Sheffer was supported by NSF award DMS-1802059. Joshua Zahl was supported by an NSERC Discovery Grant.

\section{Preliminaries}\label{sec:prelim}

\subsection{Varieties} \label{ssec:AlgGeom}
We now briefly recall standard definitions and results involving affine varieties.  For more information, see for example \cite{BCR98,Harris}. Let $\FF$ be a field; in practice we will only be interested in the fields $\RR$ and $\CC$. The \emph{variety} defined by the polynomials $f_1,\ldots, f_k \in \FF[x_1,\ldots,x_d]$ is the set
\[ \vb(f_1,\ldots,f_k) = \left\{p\in \FF^d\ :\ f_1(p)=0,\ldots, f_k(p) = 0 \right\}. \]
We say that a set $U\subset \FF^d$ is a variety if there exist polynomials $f_1,\ldots, f_k \in \FF[x_1,\ldots,x_d]$ such that $U = \vb(f_1,\ldots,f_k)$. If each of these polynomials has degree at most $D$ then we say that $U$ is \emph{is defined by polynomials of degree at most $D$}. In particular, lines, planes, and hyperplanes are defined by polynomials of degree at most one. We call varieties of this type \emph{flats}, or $k$-flats when we wish to emphasize the dimension. Note that if $U$ is defined by polynomials of degree at most $D$, then it is also defined by polynomials of degree at most $D^\prime$ for every $D^\prime\geq D$. 

If $U$ is a variety, a \emph{proper subvariety} of $U$ is a proper subset of $U$ that is also a variety. 
A variety $U$ is \emph{reducible} if it can be expressed as the union of two proper subvarieties of $U$.
Otherwise, $U$ is \emph{irreducible}.
Every variety $U$ can be uniquely expressed as a union of irreducible varieties, none of which is contained in another.
These subvarieties are the \emph{irreducible components} of $U$.

If $X\subset \FF^d$, the \emph{Zariski closure} of $X$, denoted $\overline{X}$, is the smallest variety in $\FF^d$ that contains $X$. In particular, every variety in $\FF^d$ that contains $X$ must also contain $\overline{X}$.

\parag{Dimension.}
If $U\subset\FF^d$ is an irreducible variety, we define the dimension of $U$ to be the smallest integer $k$ for which there exists a sequence 
\[ U_0 \subset U_1 \subset U_2 \subset \cdots \subset U_k = U. \]
Here, all the containments are proper and all the $U_j$ are irreducible. If $U$ is reducible, we define its dimension to be the maximum dimension of its irreducible components. We will write $\dim U$ to denote the dimension of $U$, or sometimes $\dim_{\FF} U$ if we wish to emphasize the underlying field.

We say that a variety is \emph{equidimensional} if each irreducible component has the same dimension. 
We define a \emph{curve} to be an equidimensional variety of dimension one, a \emph{surface} to be an equidimensional variety of dimension two, and a \emph{hypersurface} to be an equidimensional variety of co-dimension one. 

\parag{Regular and singular points.}
Let $U \subset \FF^d$ be an equidimensional variety of dimension $d^\prime$.
Let $\BI(U)$ be the ideal of polynomials in $\FF[x_1,\ldots,x_d]$ that vanish on every point of $U$.
Let $f_1,\ldots,f_\ell$ be polynomials that generate $\BI(U)$.
We say that $p\in U$ is a \emph{regular} point of $U$ if
\begin{equation}
\operatorname{rank}\left[\begin{array}{c}\nabla f_1(p)\\ \vdots \\ \nabla f_\ell(p)\end{array}\right]=d-d^\prime.
\end{equation}

We define $U_{\reg}$ be the set of regular points of $U$. If $U\subset\RR^d$ and $p\in U_{\reg}$, we define the tangent space $T_pU$ to be the span of $\nabla f_1(p),\ldots,\nabla f_{\ell}(p)$. 

If $p\in U$ is not a regular point of $U$, then $p$ is a \emph{singular} point of $U$. We denote this set by $U_{\sing}$. The following lemma says that most points of $U$ are regular points. 

\begin{lemma} \label{le:singular}
Let $U\subset\FF^d$ be a variety defined by polynomials of degree at most $D$.
Then $U_{\sing}$ is a variety of dimension strictly smaller than $\dim U$ that is defined by polynomials of degree $O_{d,D}(1)$.
\end{lemma}
A proof of Lemma \ref{le:singular} can be found in Section 2.2 of \cite{SSZ16}. See also Proposition 4.4 of \cite{ST11}.

\subsection{Polynomial partitioning and real algebraic geometry}
As discussed in the introduction, the polynomial partitioning theorem plays an important role in the proof of Theorem \ref{th:RichC}. In addition to Theorem \ref{th:partitionPts}, we also need the following generalization that was proved by Guth in \cite{Guth15b}.

\begin{theorem}[Polynomial partitioning for varieties]\label{th:partition}
Let $\vrts$ be a set of $n$ varieties in $\RR^d$, each of dimension $k$ and defined by polynomials of degree at most $E$. Then for each integer $D\ge 1$ there exists a nonzero polynomial $f\in \RR[x_1,\ldots,x_d]$ of degree at most $D$, such that each connected component of $\RR^d\setminus \vb(f)$ intersects $O_E(n/D^{d-k})$ varieties from $\vrts$.
\end{theorem}

In the arguments that follow, we will study incidence problems involving configurations of points and varieties. By multiplying the partitioning polynomials from Theorem \ref{th:partitionPts} and \ref{th:partition}, we obtain a partitioning polynomial that is simultaneously adapted to both sets.

\begin{corollary} \label{co:partitionCombined}
Let $\pts$ be a set of $m$ points in $\RR^d$ and let $\vrts$ be a set of $n$ varieties in $\RR^d$, each of dimension at most $k$ and defined by polynomials of degree at most $E$. 
Then for every $D\ge 1$ there exists a nonzero polynomial $f\in \RR[x_1,\ldots,x_d]$ of degree at most $D$, such that each connected component of $\RR^d\setminus \vb(f)$ intersects $O_E(n/D^{d-k})$ varieties from $\vrts$ and contains $O(mr^{-d})$ points from $\pts$.
\end{corollary}
\begin{remark}\label{rem:coDimOne}
One technical annoyance when working with real varieties is that if $f\in\RR[x_1,\ldots,x_d]$ is a non-zero polynomial, we need not have $\dim(\vb(f))=d-1$. Indeed, $\vb(f)$ can be empty, or it can have any dimension between $0$ and $d-1$. Luckily, this issue need not be a problem when performing polynomial partitioning. As discussed in \cite[Section A.3]{Zahl13}, we can always replace a polynomial $f\in\RR[x_1,\ldots,x_d]$ with a new polynomial $g$ (possibly of lower degree) so that $\vb(f)\subset \vb(g)$ and $\vb(g)$ is equidimensional and has codimension one. 
\end{remark}

To make use of Theorems \ref{th:partitionPts} and \ref{th:partition}, we need to control the size of a number of quantities related to the partitioning polynomials described above. In particular, we need to bound the number of connected components in $U\backslash \vb(f)$, where $U$ is a (not necessarily proper) subvariety of $\RR^d$. 

\begin{lemma}[Warren's theorem on a variety \cite{BB12}] \label{th:VarietyPartComponents}
Let $U \subset \RR^d$ be a variety of dimension $d'$ defined by polynomials of degree at most $E$.
Let $f\in \RR[x_1,\ldots,x_d]$ be a polynomial of degree $D$.
Then $U \setminus \vb(f)$ has $O_{d}(D^{d'}E^{d-d'})$ connected components.
\end{lemma}

\begin{lemma}[irreducible components of a variety] \label{lem:comps}
Let $U\subset \RR^d$ be a variety defined by polynomials of degree at most $D$. Then $U$ has $O_{d}(D^d)$ irreducible components. Each of these components is defined by polynomials of degree $O_{D,d}(1)$. 
\end{lemma}

When $d = 2$ the expression $O_{d}(D^2)$ from Lemma \ref{lem:comps} can be sharpened somewhat. First, B\'ezout's inequality controls the number of intersection points between two plane curves that do not share a common component.
\begin{lemma}[B\'ezout's inequality]\label{lem:Bezout}
Let $f$ and $g$ be bivariate polynomials that do not share a common factor. Then $\vb(f,g)$ contains at most $(\deg f)(\deg g)$ points.
\end{lemma}
Second, Harnack's inequality control the number of connected components of a plane curve.
\begin{lemma}[Harnack's theorem]\label{lem:Harnack}
Let $f$ be a bivariate polynomial of degree $D$. Then $\vb(f)$ has at most $(D-1)(D-2)/2 + 1\leq D^2$ connected components. 
\end{lemma}
\begin{corollary}\label{isolatedPtsInR2}
Let $U\subset\RR^2$ be a zero-dimensional variety defined by polynomials of degree at most $D$. Then $|U|\leq D^2$.
\end{corollary}
\begin{proof}
Write $U = \vb(f_1,\ldots,f_k)$, where each of $f_1,\ldots,f_k$ has degree at most $D$. Without loss of generality we can suppose that each polynomial is squarefree, and no two polynomials share a common factor. If $k=1$ then the result follows by applying Lemma \ref{lem:Harnack} to $f_1$. If $k\geq 2$ then $U \subset \vb(f_1,f_2)$, and the result follows by applying Lemma \ref{lem:Bezout} to $f_1$ and $f_2$. 
\end{proof}

We also require the following corollary of B\'ezout's inequality (for example, see \cite{GK10}).
\begin{corollary}\label{co:BezoutLines}
Let $f,g\in \CC[x_1,x_2,x_3]$ have degrees $k$ and $m$, respectively. If the intersection $\vb(f)\cap \vb(g)$ contains more than $km$ lines then $f$ and $g$ have a common factor.
\end{corollary}

\subsection{Orthogonal projections of varieties} \label{ssec:Projections}
For integers $0<e<d$, the Grassmannian $\gr(e, \FF^d)$ is the set of all $e$-dimensional linear subspaces of $\FF^d$. When $e = d-1$, we can identify each non-zero vector $v\in\FF^d$ with the orthogonal subspace $v^\perp\in \gr(d-1, \FF^d)$. In particular, if $V\subset\FF^d$ is a proper linear subspace of $\FF^d$, then $\{v^\perp\colon v\in V\}$ is a proper subset of $\gr(d-1, \FF^d)$ (indeed, if we give  $\gr(d-1, \FF^d)$ the structure of a variety then the above set is a proper subvariety of $\gr(d-1, \FF^d)$).

We associate every element $V\in \gr(e, \FF^d)$ with the orthogonal projection of $\FF^d$ onto the corresponding $e$-dimensional space. We denote this projection as $\pi_V:\FF^d \to \FF^e$. If $U\subset\FF^d$ is a variety and $V\in \gr(e, \FF^d)$, then $\pi(U)$ need not be a variety. However, the next lemma shows that $\pi(U)$ is contained in a variety whose complexity is not too much larger than that of $U$.

\begin{lemma} \label{le:projDeg}
Let $\FF$ be the field $\RR$ or $\CC$. Let $0 < e < d$, let $V\in \gr(e, \FF^d)$, and let $U\subset \FF^d$ be a variety defined by polynomials of degree at most $D$.
Then $\overline{\pi_V(U)}$ is a variety of dimension at most $\dim(U)$ that is defined by polynomials of degree $O_{d,D}(1)$. 
\end{lemma}
\begin{proof}[Proof sketch]
When $\FF=\CC$, then this is proved in \cite[Theorem 3.16]{Harris} or \cite[Chapter 2.6, Theorem 6]{Matsu80}.
These two theorems state that $\pi_V(U)$ is a \emph{constructible set} (a set defined by a Boolean combination of algebraic equalities and non-equalities) whose Zariski closure has dimension at most $\dim(U)$. 
The proof is constructive and thus provides an upper bound on the degree of the polynomials that define the constructible set. 

When $\FF=\RR$, this is a consequence of the (effective) Tarski-Seidenberg theorem. See for example \cite[Section 2]{BCR98}. The Tarski-Seidenberg theorem states that $\pi_V(U)$ is a \emph{semi-algebraic set} of complexity $O_{d,D}(1)$ whose Zariski closure has dimension at most $\dim(U)$. 
\end{proof}

\subsection{Containment degree} \label{ssec:contDegree}

Let $0< e<d$ and let $X\subset\RR^d$.  
We define the \emph{containment degree} $\operatorname{ContDeg}(X,e)$ to be the smallest integer $D$ such that there exists a variety in $\RR^d$ of dimension $e$ defined by polynomials of degree at most $D$ that contains $X$. If $e\ge \dim \overline{X}$ then $\operatorname{ContDeg}(\Gamma,e)$ is well-defined and finite. If $e< \dim \overline{X}$ then no variety of dimension $e$ can contain $X$, and we set $\operatorname{ContDeg}(X,e)=\infty.$ Note that $\operatorname{ContDeg}(X,e)=\operatorname{ContDeg}(\overline{X},e)$ for every set $X\subset\RR^d$ and every $0<e<d$.

We say that a projection $\pi_V:\RR^d \to \RR^e$ is \emph{degenerate with respect to} $X$ if there is an index $1\leq t< e$ such that
\begin{equation}\label{eq:degenProj}
\operatorname{ContDeg}(X,t)> \operatorname{ContDeg}(\overline{\pi_V(X)},t).
\end{equation}
If $\pi_V$ is not degenerate then we call it non-degenerate. If $\pi_{V}\in \gr(d',\RR^{d})$ is non-degenerate with respect to $X$ and $\pi_{V'}\in\gr(d'',\RR^{d'})$ is non-degenerate with respect to $\pi_V(X)$, then $\pi_{V'}\circ \pi_V\in\gr(d'',\RR^d)$ is non-degenerate with respect to $X$. 
This allows us to construct a non-degenerate projection from $\RR^d\to\RR^e$ by composing a sequence of projections from $\RR^{d}\to\RR^{d-1}$,\ $\RR^{d-1}\to\RR^{d-2}$, and so on. 
First, we show that most projections do not decrease the containment degree.
For $v\in \RR^d$ that is not the origin, we denote by $\pi_v\colon\RR^d\to\RR^{d-1}$ the orthogonal projection in direction $v$.
 
\begin{lemma}\label{lem:goodProjectionDirectionCodim1}
Let $X\subset\RR^d$ be a set whose Zariski closure has dimension at most $d-2$. Then there are $d-2$ proper linear subspaces $V_1,\ldots,V_{d-2}\subset \RR^d$ such that for all $v\in\RR^d\backslash (V_1\cup \cdots \cup V_{d-2})$ and $0< t<d-1$, 
\begin{equation}\label{eq:nonDegenProj}
\operatorname{ContDeg}(X,t)\le \operatorname{ContDeg}(\pi_v(X),t).
\end{equation}
\end{lemma}
\begin{proof}
We will show that for each $0<t<d-1$, the set of vectors $v\in\RR^d$ for which \eqref{eq:nonDegenProj} fails must be contained in a proper linear subspace of $\RR^d$ (note that the subspaces for different choices of $t$ might be different). We will prove this by contradiction. Suppose that there exist an index $0< t<d-1$ and $d$ linearly independent directions $v_1,\ldots,v_d$ such that \eqref{eq:nonDegenProj} fails with $v = v_j$ for each $j$. Applying an invertible linear transformation (such a transformation leaves the containment degree unchanged), we may assume that $v_j$ is the $j$--th basis vector. 

Let $D = \operatorname{ContDeg}(X,t)$. Since \eqref{eq:nonDegenProj} fails for each vector $v_j$, for each index $j$ there is a collection of polynomials $\mathcal{F}_j$, each of degree at most $D-1$ and independent of the variable $x_j$, such that $X\subset \bigcap_{f\in\mathcal{F}_j}\vb(f)$. We claim that $U = \bigcap_j \bigcap_{f\in\mathcal{F}_j}\vb(f)$ has dimension at most $t$. Since $X\subset U$, this would imply that $\operatorname{ContDeg}(X,t)\le D-1$, which would contradict the definition of $D$ and complete the proof.

Indeed, for each $j$, the dimension of $\bigcap_{f\in\mathcal{F}_j}\vb(f)$ is at most $t+1$. If there exists $1<j<d-1$ such that one of the polynomials in $\mathcal{F}_j$ includes $x_1$, then $\dim(U)\leq \dim(\bigcap_{f\in\mathcal{F}_1\cup\mathcal{F}_j}\vb(f))\leq t$. Thus, in this case we are done. Otherwise, we repeat the above argument for $x_2,x_3,\ldots,x_d$. If none of these events occur, then each polynomial in $\bigcup_j\mathcal{F}_j$ must be constant, which is impossible. 
\end{proof}

Let $\mathcal{X}$ be a family of subsets of $\RR^d$. We say that a projection $\pi:\RR^d \to \RR^e$ is non-degenerate with respect to $\mathcal{X}$ if it is non-degenerate with respect to each set $X\in\mathcal{X}$.

\begin{lemma}\label{lem:goodProjectionSet} $\quad$\\
(a) Let $0< e< d$. Let $\mathcal{X}$ be a finite family of subsets of $\RR^d$, where each $X\in \mathcal{X}$ satsifies $\dim(\overline{X})<e$. Then there exists a projection $\pi\in\gr(e,\RR^d)$ that is non-degenerate with respect to $\mathcal{X}$. \\
(b) When all sets in $\mathcal{X}$ satisfy $\dim(\overline{X})<e/2$, we may assume that $\pi$ also has the following property. For every pair of sets $X,X^\prime\in\mathcal{X}$, we have $|\overline{X}\cap \overline{X^\prime}|=|\overline{\pi(X)}\cap \overline{\pi(X^\prime)}|$. 
\end{lemma}
\begin{proof}
(a) For each $X\in\mathcal{X}$, let $V_{1,X},\ldots,V_{d-2,X}$ be the proper linear subspaces described in Lemma \ref{lem:goodProjectionDirectionCodim1}, and select $v_d\in \RR^d\backslash\bigcup_{X\in\mathcal{X}}\{V_{1,X},\ldots,V_{d-2,X}\}$. By Lemma \ref{lem:goodProjectionDirectionCodim1}, we have
$$
\operatorname{ContDeg}(X,t)\le \operatorname{ContDeg}(\pi_{v_d}(X),t)
$$
for each $0< t<d-1$ and each $X\in\mathcal{X}$. Define $\mathcal{X}_{d-1} = \{ \pi_{v_d}(X)\colon X\in\mathcal{X}\}$; this is a family of subsets of $\RR^{d-1}$, each of which satisfies $\dim(\overline{X})<e$. 

Repeat this process to select a vector $v_{d-1}\in\RR^{d-1}$ so that
$$
\operatorname{ContDeg}(X,t)\le \operatorname{ContDeg}(\pi_{v_{d-1}}(X),t)
$$
for each $0< t<d-2$ and each $X\in\mathcal{X}_{d-1}$. Continuing this process, we obtain a sequence of vectors $v_d\in\RR^d, v_{d-1}\in\RR^{d-1},\ldots, v_{e+1}\in\RR^{e+1}$ and a sequence of families $\mathcal{X}_{d-1},\ldots,\mathcal{X}_{e}$. Define $\pi$ to be the composition $\pi_{v_{e+1}}\circ\pi_{v_{e+2}}\circ\cdots\circ v_d$.



(b) We repeat the proof of part (a) with a small addition. At each step when we select a vector $v_{d-j}\in\RR^{d-j}$, we choose the vector so that for every pair of sets $X,X^\prime\in\mathcal{X}_{d}$, we have $|\overline{X}\cap \overline{X^\prime}|=|\overline{\pi_{v_{d-j}}(X)}\cap \overline{\pi_{v_{d-j}}(X^\prime)}|$. Since $\dim(\overline{X})<e/2$ and $\dim(\overline{X^\prime})<e/2,$ the set of vectors $v\in\RR^{d-j}$ for which the above inequality fails is contained in a proper sub-variety of $\RR^d$.
\end{proof}

\subsection{Ruled surface theory}

Let $\FF$ be a field of characteristic 0, let $S\subset\FF^3$ be an irreducible surface, and let $D\geq 1$ be an integer. We say that $S$ is \emph{doubly ruled by curves} defined by polynomials of degree at most $D$ if the following holds. 
There is a proper subvariety $T\subset S$ such that for all $p\in S\backslash T$, at least two irreducible curves $\gamma,\gamma^\prime\subset S$ contain $p$ (and are defined by polynomials of degree at most $D$).

Let $f\in \FF[x,y,z]$ have degree at most $E$ and let $S=\vb(f)$. Let $\gamma\subset S$ be an irreducible curve defined by polynomials of degree at most $D$. We say that $\gamma$ is an \emph{exceptional curve} if at least $C_DE$ points $p\in\gamma$ satisfy the following. There is an irreducible curve $\gamma^\prime\neq \gamma$ defined by polynomials of degree at most $D$ such that $p\in\gamma^\prime\subset S$. 
Note that the definition of exceptional curves depends on the choice of $C_D$. 
We set this constant to be as in the following lemma (see \cite{GZ15a}).

\begin{lemma}\label{lem:tooManyExceptionsImpliesDoublyRuled}
For each $D\geq 1$, there is a constant $C_D$ such that the following holds. Let $f\in \FF[x,y,z]$ have degree $E$ and let $S=\vb(f)$. If $S$ contains more than $C_DE^2$ exceptional curves defined by polynomials of degree at most $D$, then $S$ is doubly ruled by curves defined by polynomials of degree at most $D$. 
In this case $E = O_D(1)$. Furthermore, if $D = 1$ then $E\leq 2$.  
\end{lemma}
\begin{remark}
Lemma \ref{lem:tooManyExceptionsImpliesDoublyRuled} also holds in characteristic $p$, provided $E$ is not too large compared to $D$ and $p$.
\end{remark}

\section{The $r=2$ and $r=3$ case: Koll\'ar's bound}\label{sec:KollarBd}
As discussed in the introduction, the proof of Theorem \ref{th:numberOfRRichPoints} when $r=2$ is purely algebraic, and it extends to other fields. The main tool is the following theorem due to Koll\'ar \cite{Kollar15}. See also \cite{GZ15a}.

\begin{theorem}\label{th:Kollar}
Let $\lines$ be a set of $n$ lines in $\CC^3$, such that every plane and degree two surface contains at most $n^{1/2}$ lines from $\lines$.
Then 
\[ |\pts_2(\lines)|=O(n^{3/2}). \]
\end{theorem}

To use Theorem \ref{th:Kollar} we will also need a simple result that controls the number of planes and degree two varieties that contain many lines.
\begin{lemma}\label{lem:FewRichSurfaces}
Let $\lines$ be a set of $n$ curves in $\FF^3$ and let $A\geq 2E^2n^{1/2}$. Let $\surfs$ be a set of irreducible surfaces in $\FF^3$, each defined by polynomials of degree at most $E$ and contains at least $A$ curves from $\lines$. Then $|\surfs|\leq 2nA^{-1}$.
\end{lemma}
\begin{proof}
Let $\surfs=\{S_1,\ldots,S_k\}$. 
The intersection of two surfaces defined by polynomials of degree $E$ and sharing no common components contains at most $E^2$ curves (for example, see \cite[Theorem 5.7]{GZ15a}).
Since the surfaces in $\surfs$ are irreducible, each pair of distinct surfaces can contain at most $E^2$ common curves from $\lines$. Thus for each $j\geq 1$ we have
\begin{equation*}
\begin{split}
\Big|\bigcup_{\ell=1}^j \lines_{S_\ell}\Big|& \geq \sum_{\ell=1}^j |\lines_{S_\ell}| - \sum_{1\leq \ell < m \leq j}|\lines_{S_\ell} \cap \lines_{S_m}| \geq A\cdot j - E^2 \cdot j(j-1)/2.
\end{split}
\end{equation*} 

Assume for contradiction that $|\surfs|> 2n/A$. Set $j = \lceil 2nA^{-1}\rceil$ and note that 
\begin{equation*}
E^2\cdot (j-1)/2 \leq E^2nA^{-1} \leq n^{1/2}/2 \leq A/4.
\end{equation*}
Thus 
\begin{equation*}
\Big|\bigcup_{\ell=1}^j \lines_{S_\ell}\Big|\geq (A - E^2(j+1)/2 )j \geq (3/4)A \cdot j > n,
\end{equation*}
which is a contradiction.
\end{proof}

To control the number of $r$-rich points determined by a set of lines in $\CC^3$, we use the complex variant of the Szemer\'edi--Trotter theorem \cite{Toth15, Zahl14, ST11}.

\begin{theorem} \label{th:complexST}
Let $\lines$ be a set of $n$ lines in $\CC^2$.
Then for every $r\ge 2$, we have
\[|\pts_r(\lines)|=O\left(\frac{n^{2}}{r^3}+\frac{n}{r}\right). \]
\end{theorem}

With these tools, we can now prove the ``algebraic'' part of Theorem \ref{th:RichC}.

\begin{proof}[Proof of Theorem \ref{th:RichC}, $r=2$ and $r=3$ case]
Let $\surfs_1$ be the set of all planes that contain more than $2n^{1/2}$ lines from $\lines$, and let $\surfs_2$ be the set of all irreducible degree two surfaces that contain more than $8n^{1/2}$ lines from $\lines$. Lemma \ref{lem:FewRichSurfaces} implies that $|\surfs_1|,|\surfs_2| =O(n^{1/2})$. 

Let $\lines^\prime=\lines\backslash\bigcup_{S\in\surfs_1\cup \surfs_2}\lines_S$.
By Theorem \ref{th:Kollar},
$$
|\pts_2(\lines^\prime)|=O(n^{3/2}).
$$
(Theorem \ref{th:Kollar} requires at most $n^{1/2}$ lines in a surface, while we have at most $8n^{1/2}$. To address this we slightly increase $n$ by adding generic lines to $\lines'$.)
A line $L\in\lines$ intersects each surface $S\in\surfs_1\cup\surfs_2$ that does not contain $L$ in at most two points.
Thus, at most $2|\lines| |\surfs_1\cup\surfs_2|=O(n^{3/2})$ points $p\in\CC^3$ are incident to a line $L\in \lines$ and a surface $S\in\surfs_1\cup\surfs_2$ satisfying $L\not\in \lines_S$. We conclude that
$$
\big|\pts_2(\lines)\ \backslash\bigcup_{S\in\surfs_1\cup\surfs_2}\pts_2(\lines_S)\big|=O(n^{3/2}).
$$
Setting $\surfs=\surfs_1\cup\surfs_2$ concludes the proof of Theorem \ref{th:RichC} when $r=2$. 

We now consider the case of $r=3$. 
By considering every type of irreducible quadratic surface in $\RR^3$, we note that at most one point in such a quadratic $S$ is incident to three lines that are contained in $S$. 
This implies that $\sum_{S\in\surfs_2}|\pts_3(\lines_S)| = O(n^{1/2})$.
There are three other ways for a point $p$ to be in $\pts_3(\lines)$: (i) $p\in\pts_3(\lines^\prime)$, (ii) $p\in\pts_3(\lines_S)$ for some $S\in\surfs_1$, and (iii) $p$ is the intersection of a line $L \in\lines$ and a surface $S\in\surfs_1\cup\surfs_2$ satisfying $L\not\in\lines_S$.
Since $\pts_3(\lines^\prime) \subset \pts_2(\lines^\prime)$ and line--surface intersections were bounded above, we conclude that 
$$
\big|\pts_3(\lines)\ \backslash\bigcup_{S\in\surfs_1}\pts_3(\lines_S)\big|=O(n^{3/2}).
$$
Setting $\surfs=\surfs_1$ concludes the proof of Theorem \ref{th:RichC} when $r=3$. 
\end{proof}

One of the hypotheses of Theorem \ref{th:RichC} is that $r\leq 2n^{1/2}$. When $r>2n^{1/2}$, a much simpler argument gives a tight bound on the number of $r$-rich points. This bound has little to do with the geometry of lines in $\CC^3$. It only relies on the fact that each pair of lines intersect in at most one point. We record this observation below.
\begin{lemma}\label{lem:rRichPtsBigR}
Let $\mathcal{X}$ be a family of subsets of a ground set $X$, and suppose each pair of sets in $\mathcal{X}$ intersect in at most one element. Let $r\geq 2|\mathcal{X}|$ and let $\pts\subset X$ be the set of elements contained in at least $r$ sets from $\mathcal{X}$. Then $|\pts|\leq 2|\mathcal{X}|r^{-1}$. 
\end{lemma} 
The proof of Lemma \ref{lem:rRichPtsBigR} is very similar to the proof of Lemma \ref{lem:FewRichSurfaces} when $E=1$. We do not repeat this proof.

\section{Guth's structure theorem: curves in $\RR^d$}\label{sec:GuthStructureTheorem}

In \cite{Guth15a}, Guth proved Theorem \ref{th:RichRGuth}, which is a structure theorem for sets of lines in $\RR^3$. As discussed in the introduction, Guth proves his result by induction on the number of lines.
He used polynomial partitioning (Theorem \ref{th:partitionPts}) to break the collection of lines into several significantly smaller sub-collections. 

Guth stated his result for lines, which are the primary objects of interest when studying the distinct distances problem in the plane. However, his proof relies only on few properties that are specific to lines:
\begin{enumerate}[noitemsep,topsep=1pt]
\item A special case of Lemma \ref{lem:FewRichSurfaces} when $K=\RR$ and $E=1$.
\item Lemma \ref{lem:rRichPtsBigR}.
\item A variant of B\'ezout's theorem: Let $f$ be a polynomial of degree $D$. Then a line not contained in $\vb(f)$ intersects $\vb(f)$ in at most $D$ points.
\item The Szem\'eredi-Trotter theorem: Any set of $n$ lines in $\RR^2$ determines $O(n^2r^{-3}+ nr^{-1})$ $r$-rich points, for each $r\geq 2$. 
\end{enumerate}

We formulated Lemmas \ref{lem:FewRichSurfaces} and \ref{lem:rRichPtsBigR} in a way that holds for arbitrary curves in $\RR^3$. 
The third item could be replaced with Lemma \ref{th:VarietyPartComponents}: If $\gamma\subset\RR^3$ is an irreducible curve defined by polynomials of degree at most $E$ and $f\in\RR[x,y,z]$ of degree $D$ satisfies $\gamma\not\subset \vb(f)$, then $|\gamma\cap\vb(f)|=O_E(D)$. Finally, the Szem\'eredi-Trotter theorem has the following generalization \cite{CEGSW90}.
\begin{lemma}\label{lem:CEGSW}
Let $\curves$ be a set of $n$ irreducible curves in $\RR^2$, each defined by polynomials of degree at most $E$.
For every pair of distinct points $p,q\in \RR^2$, at most $M$ curves from $\curves$ are incident to both $p$ and $q$. Then for each $r\geq 2$, the number of $r$-rich points determined by $\curves$ is $O_{E,M}(n^2r^{-3}+nr^{-1})$.
\end{lemma}

Keeping these minor changes in mind, Guth's theorem can be restated as a structure theorem for curves in $\RR^3$. 
\begin{lemma}\label{lem:guthStructureR3}
For every $\eps>0$ and $E,M\geq 1$, there is a constant $C$ such that the following holds. Let $\curves$ be a set of $n$ irreducible curves in $\RR^3$, each defined by polynomials of degree at most $E$. For every pair of distinct points $p,q\in \RR^3$, at most $M$ curves from $\curves$ are incident to both $p$ and $q$. Let $2\leq r\leq 2n^{1/2}$ and let $r^\prime = \lceil 9r/10\rceil$. Then there exists a set $\surfs$ of surfaces in $\RR^3$ with the following properties.
\begin{itemize}[noitemsep]
\item Every surface in $\surfs$ is defined by polynomials of degree at most $C$. 
\item Every plane $W\in \surfs$ contains at least $rn^{1/2+\eps}$ curves from $\curves$. 
\item $|\surfs|\le 2n^{1/2-\eps}r^{-1}$.
\item $|\pts_r(\curves)\setminus \bigcup_{W\in \surfs}\pts_{r^\prime}(\curves_W)|\leq C n^{3/2+\eps}r^{-2}$.
\end{itemize}
\end{lemma}


Using Lemma \ref{lem:goodProjectionSet}, we can extend Lemma \ref{lem:guthStructureR3} to curves in $\RR^d$. 

\begin{lemma}\label{th:RichR}
For every $\eps>0$, $d\geq 3$ and $E,M\geq 1$, there is a constant $C$ such that the following holds. Let $\curves$ be a set of $n$ irreducible curves in $\RR^d$, each defined by polynomials of degree at most $E$. For every pair of distinct points $p,q\in \RR^d$, at most $M$ curves from $\curves$ are incident to both $p$ and $q$. Let $2\leq r\leq 2n^{1/2}$ and let $r^\prime = \lceil 9r/10\rceil$. Then there exists a set $\surfs$ of surfaces in $\RR^d$ with the following properties.
\begin{itemize}[noitemsep]
\item Every surface in $\surfs$ is defined by polynomials of degree at most $C$. 
\item Every plane $W\in \surfs$ contains at least $rn^{1/2+\eps}$ curves from $\curves$. 
\item $|\surfs|\le 2n^{1/2-\eps}r^{-1}$.
\item $|\pts_r(\curves)\setminus \bigcup_{W\in \surfs}\pts_{r^\prime}(\curves_W)|\leq C n^{3/2+\eps}r^{-2}$.
\end{itemize}
\end{lemma}
\begin{proof}
Use Lemma \ref{lem:goodProjectionSet} to find a projection $\pi:\RR^d\to\RR^3$ that is non-degenerate with respect to the set $\mathcal{X}=\{\bigcup_{\gamma\in\curves^\prime}\gamma\colon\curves^\prime\subset\curves\}$. Set $\curves_{\RR^3} = \{\overline{\pi(\gamma)}\ :\ \gamma\in \curves\}$. By part (b) of Lemma \ref{lem:goodProjectionSet}, for every pair of distinct points $p,q\in \RR^r$, at most $M$ curves from $\curves_{\RR^3}$ are incident to both $p$ and $q$. We apply Lemma \ref{lem:guthStructureR3} with $\curves_{\RR^3}$, to obtain a set $\mathcal{S}_{\RR^3}$ of surfaces, each defined by polynomials of degree at most $C = C(\eps,E)$. Since $\pi$ is non-degenerate, for each $S_{\RR^3}\in\surfs_{\RR^3}$, there is a corresponding surface $S\subset\RR^d$ defined by polynomials of degree at most $E$ with the property: For each $\gamma\in\curves$ with $\overline{\pi(\gamma)}\subset S_{\RR^3}$ we have that $\gamma\subset S$. Let $\surfs$ be the set of surfaces in $\RR^d$ that correspond to the surfaces of $\surfs_{\RR^3}$. We can verify that this set of surfaces satisfies the requirements of the lemma.
\end{proof}

\section{Complex incidence geometry inside a real hypersurface}\label{sec:cplxGeom}
\subsection{Complex lines in a real variety}
In this section we study the set of complex lines that can be contained in a real variety in $\RR^6$. Throughout this section, we identify $\CC$ with $\RR^2$ using the map $(x+iy)\mapsto (x,y)$. We similarly identify $\CC^3$ with $\RR^6$. We call a subset of $\RR^6$ a \emph{complex line} if it is the image of a complex line in $\CC^3$ under this identification. We call a subset of $\RR^6$ a \emph{complex plane} if it is the image of a complex plane under this identification. We often abuse notation and refer to a complex line as a subset of $\CC^3$ or of $\RR^6$. 

The set of all complex lines in $\CC^3$ can be identified with an algebraic structure called a quasi-projective variety. For our purposes, however, it will be simpler to restrict attention to a large subset of the set of complex lines. We say that a line $L \subset\CC^3$ is \emph{standard} if it is not parallel to the complex $z_2z_3$ plane. Every standard line  $L$ can be expressed in the form $(0, a,b)+t\cdot (1,c,d)$ with fixed $a,b,c,d\in\CC$ and a parameter $t\in \CC$. We define 
$$
G(L) = (\operatorname{Re}(a),\operatorname{Im}(a),\operatorname{Re}(b),\operatorname{Im}(b),\operatorname{Re}(c),\operatorname{Im}(c),\operatorname{Re}(d),\operatorname{Im}(d)).
$$
Note that $G$ is bijection from the set of standard complex lines to $\RR^8$. 

When working with standard lines, it will be useful to define the map
\begin{align*}
\phi(a_1,a_2,b_1,b_2,&c_1,c_2,d_1,d_2,s,t)\\
&= (s, t, a_1 + sc_1 -tc_2, a_2 + sc_2 + tc_2, b_1 + sd_1 - td_2, b_2 + sd_2 + td_1).
\end{align*}
That is, $\phi(a_1,a_2,b_1,b_2,c_1,c_2,d_1,d_2,s,t)$ is the image of the point $(0,a_1+ia_2,b_1+ib_2)+(s+it)(1,c_1+ic_2,d_1+id_2)$ under the identification of $\CC^3$ with $\RR^6$.

For a variety $U\subset\RR^6$, we define $L(U)$ to be the set of standard complex lines contained in $U$. Abusing notation slightly, we define $G(U) = G(L(U))\subset\RR^8$.

The following observation plays a crucial role in the arguments that follow. If $U\subset\RR^6$ is a variety, $p\in U_{\reg}$, and $H\subset U$ is a plane that contains $p$, then $H$ must be contained in the tangent space $T_pU$. If $L\subset U$ is a complex line that contains $p$ and is contained in $U$, then more is true. In addition to $L$ being contained in $T_pU$, it must also be contained in a certain subspace of $T_pU$ that is compatible with the complex structure of $L$. To make this precise we define the operator $J\colon \RR^6\to\RR^6$ as 
\begin{equation}
J(x_1,y_1,x_2,y_2,x_3,x_3) = (-y_1,x_1,-y_2,x_2,-y_3,x_3).
\end{equation}
If we identify $\RR^6$ with $\CC^3$, then $J$ corresponds to multiplication by $i$. For $p\in U_{\reg}$ we define the \emph{complex tangent space}
\begin{equation}
V_p(U) = T_p(U) \cap J(T_pU). 
\end{equation}
This is (a translate of) the largest complex linear subspace of $\CC^3$ that is contained in $T_pU$. Observe that $V_pU$ must have an even dimension. In particular, if $U$ is a proper subvariety of $\RR^6$ then $V_pU$ has dimension at most 4. 

With these definitions, we can begin to study the set of complex lines contained in a real variety.
\begin{lemma}\label{lem:varietyOfLinesInZ}
Let $U\subset\RR^6$ be a variety defined by polynomials of degree at most $D$. Then $G(U)$ is a variety defined by polynomials of degree at most $D$. 
\end{lemma}
\begin{proof}
Let $f_1,\ldots,f_k$ be polynomials of degree at most $D$ such that $U= \vb(f_1,\ldots,f_k)$. For each index $j$, consider the polynomial
\begin{equation}
\begin{split}
(a_1,a_2,b_1,b_2,c_1,c_2,d_1,d_2,s,t)&\mapsto f_j( \phi(a_1,a_2,b_1,b_2,c_1,c_2,d_1,d_2,s,t))\\
&=\sum_{0\leq u\leq v\leq D} Q_{j,u,v}(a_1,a_2,b_2,b_2,c_1,c_2,d_1,d_2)s^u t^v.
\end{split}
\end{equation}

A standard line $L$ with $G(L) = (a_1,a_2,b_1,b_2,c_1,c_2,d_1,d_2)$ vanishes identically on $\vb(f_j)$ if and only if $Q_{j,u,v}(a_1,a_2,b_1,b_2,c_1,c_2,d_1,d_2)=0$ for each $0\leq u\leq v\leq D$. We conclude that 
\begin{equation}
G(U)=\bigcap_{j=1}^k \bigcap_{0\leq u\leq v\leq D} \vb(Q_{j,u,v}).
\end{equation}
Each of these polynomials has degree at most $D$.
\end{proof}

If $\Pi\subset\RR^6$ is a complex plane, then by Lemma \ref{lem:varietyOfLinesInZ}, the variety $G(\Pi)\subset\RR^8$ is defined by polynomials of degree one. In fact, if $\Pi$ is not parallel to the $z_2z_3$ plane then $G(\Pi)$ is a four-dimensional linear variety in $\RR^8$. Under the standard identification of $\RR^8$ with $\CC^4$, the variety $G(\Pi)$ is a complex plane. Furthermore, if $\Pi$ and $\Pi^\prime$ are complex planes in $\RR^6$, then $G(\Pi)$ and $G(\Pi^\prime)$ are either disjoint (when $\Pi$ and $\Pi^\prime$ are parallel), or they intersect at a single point (corresponding to the complex line $\Pi\cap \Pi^\prime$).

For $p\in\RR^6$, we define $G_p\subset\RR^8$ to be the (image of the) set of standard complex lines that contain $p$. Again, $G_p$ is a four-dimensional linear variety in $\RR^8$. Under the standard identification of $\RR^8$ with $\CC^4$, the set $G_p$ is a complex plane. If $p$ and $p^\prime$ are distinct, then either $G_p$ and $G_{p^\prime}$ are disjoint (when the complex line containing $p$ and $p^\prime$ is parallel to the complex $z_2z_3$ plane), or they intersect at a single point (corresponding to the complex line containing $p$ and $p^\prime$). 

For a standard complex line $L\subset\RR^6$, we define $\hair(L)$ to be the set of (images of the) standard complex lines that intersect $H$. We refer to $\hair(L)$ as the \emph{hairbrush} of $L$. A hairbrush is a six-dimensional variety in $\RR^8$ defined by polynomials of degree two.  If $L$ and $L^\prime$ are standard complex lines that intersect at a point $p$ and span the complex plane $\Pi$, then $H(L) \cap H(L^\prime) = G_p \cup G(\Pi)$. 

\begin{lemma}\label{lem:maxFiniteLinesThroughPt}
Let $U$ be a proper subvariety of $\RR^6$ defined by polynomials of degree at most $D$ and let $p\in U_{\operatorname{reg}}$. If $G_p \cap G(U)$ is finite, then it has cardinality at most $D^2$. If it is infinite, then it has dimension one or two. If $G_p \cap G(U)$ has dimension two then there is a complex plane $\Pi\subset U$ that contains $p$.
\end{lemma}
\begin{proof}
If $L$ is a complex line with $p\in L\subset U$, then $L$ must be contained in the complex plane $\Pi = V_pU$. This means that the set of all such complex lines is given by $G(U)\cap G_p\cap G(\Pi).$ This variety has dimension at most two. The dimension is two if and only if $G_p\cap G(\Pi) \subset G(U)$, in which case $\Pi\subset U$. Since $G_p\cap G(\Pi)$ is a real plane, we can think of $G(U)\cap G_p\cap G(\Pi)$ as a variety in $\RR^2$. 
By Lemma \ref{lem:varietyOfLinesInZ} this variety is defined by polynomials of degree at most $D$. By Corollary \ref{isolatedPtsInR2}, if $G(U)\cap G_p\cap G(\Pi)$ is finite then it has cardinality at most $D^2$.
\end{proof}

\begin{lemma}\label{lem:DenseInPlane}
Let $\Pi\subset\RR^6$ be a complex plane and let $U\subset G(\Pi)$ be a variety of dimension at least two. Then there does not exist a real proper subvariety $X\subset\Pi$ such that every complex line corresponding to a point of $U$ is contained in $X$.
\end{lemma}
\begin{proof}
Assume for contradiction that there exists $X\subset\Pi$ as stated in the lemma. 
Since $\dim X_{\sing}\le 2$, the set of complex lines from $G(\Pi)$ that have an infinite intersection with $X_{\sing}$ is of dimension at most one.
Similarly, the set of such lines that are contained in $X$ is of dimension at most one. 
Recall that two generic lines from $G(\Pi)$ intersect.
Combining the above, we conclude that a generic line $L_0\in G(\Pi)$ satisfies $\dim(L_0\cap X_{\reg})\leq 1,$ $\dim(L_0\cap X_{\sing})\leq 0$, and $\dim(H(L_0)\cap U)\geq 2$. 
Fix a line $L_0$ that satisfies these three properties.

We claim that every point $p\in L_0\cap X_{\sing}$ satisfies $\dim(G_p\cap U)\leq 1$. 
Indeed, if $\dim(G_p\cap U)= 2$ then $G_p\cap G(\Pi) \subset U$, which implies that the union of the lines in $U$ is $\Pi$. This contradicts the assumption about $X$ being a proper sub-variety of $\Pi$. 

For a point $p\in L_0\cap X_{\reg}$ we have that $\dim V_pX\le 2$ (since this dimension must be even). In this case, at most one complex line $L\subset X$ satisfies $p\in L$. This implies that the set of lines $L\subset X$ with $L\cap L_0\subset X_{\reg}$ is contained in a subvariety of $U$ of dimension at most one. We conclude that $\dim(H(L_0)\cap U)\leq 1$, which contradicts the definition of $L_0$.
\end{proof}

\begin{corollary}\label{cor:oneDimIntersectionWithCPlane}
Let $U\subset\RR^6$ be a variety defined by polynomials of degree at most $D$. Let $\Pi\subset\RR^6$ be a complex plane that is not contained in $U$. Then $\dim(G(U)\cap G(\Pi))\leq 1$
\end{corollary}
\begin{proof}
Suppose to the contrary that $\dim(G(U)\cap G(\Pi))\geq 2$. Then for each $w\in G(U)\cap G(\Pi)$, the line $L_w$ is contained in $U\cap \Pi$, which is a proper subvariety of $\Pi$. This contradicts Lemma \ref{lem:DenseInPlane}.
\end{proof}

Combining Lemma \ref{lem:maxFiniteLinesThroughPt} and Corollary \ref{cor:oneDimIntersectionWithCPlane}, we obtain the following.
\begin{lemma}\label{lem:exceptionalPtsInCPlane}
Let $U\subset\RR^6$ be a variety defined by polynomials of degree at most $D$. Let $\Pi\subset\RR^6$ be a complex plane that is not contained in $U$. Then there are $O(D^4)$ \emph{exceptional} points in $U\cap \Pi$. If $p\in U\cap \Pi$ is not an exceptional point, then there are at most $D^2$ complex lines $L\subset U\cap \Pi$ with $p\in L$. 
\end{lemma}
\begin{proof}
By Corollary \ref{cor:oneDimIntersectionWithCPlane}, $G(U)\cap G(\Pi)$ is of dimension at most one. We Identify $\Pi$ with $\RR^4$, considering the intersection $G(U)\cap G(\Pi)$ as a variety in $\RR^4$ defined by polynomials of degree at most $D$.  If $p\in U\cap \Pi$ satisfies $\dim (G_p\cap G(U\cap \Pi))= 1$, then we call it an exceptional point. Since $|G_p\cap G_{p^\prime}|\leq 1$ whenever $p$ and $p^\prime$ are distinct, the number of exceptional points is at most the number of irreducible components of  $G(U)\cap G(\Pi)$. By applying Lemma \ref{lem:comps} in $\RR^4$, we get that this number is $O(D^4)$. If $p$ is not an exceptional point then $G_p\cap G(U\cap \Pi)$ is finite. Since $G_p\cap G(\Pi)$ can be identified with $\RR^2$, we can use Corollary \ref{isolatedPtsInR2} to conclude that $|G_p\cap G(U\cap \Pi)|\leq D^2$. 
\end{proof}

\subsection{Real varieties ruled by complex planes}
Recall that a variety $U\subset\RR^6$ is almost ruled by complex planes if for each regular point $p\in U_{\operatorname{reg}}$, there is a complex plane $\Pi\subset U$ that contains $p$. If $U=\RR^6$, then $U$ is almost ruled by complex planes; this situation is not very interesting. If $U$ is a non-empty proper subvariety of $\RR^6$ that is almost ruled by complex planes, then $\dim(U) \geq 4$. If $\dim(U)=4$ and $U$ is ruled by complex planes then $U$ must be a finite union of complex planes. Finally, if $\dim U=5$ and $U$ is ruled by complex planes, then for each $p\in U_{\reg}$, the complex plane $V_pU$ is the unique complex plane satisfying $p\in \Pi\subset U$. 

\begin{lemma}\label{lem:varietyRuledPlanesOrFewPlanes}
Let $U=\vb(f)\subset\RR^6$ be an irreducible hypersurface with $\deg f\leq D$. Then either $U$ is almost ruled by complex planes or $U$ contains at most $2D^2(D-1)$ complex planes. 
\end{lemma}
\begin{proof}
Without loss of generality we can assume that $\nabla f\neq 0$ on $U_{\reg}$. Let $p\in U$ and suppose that $\nabla f\neq 0$. For each index $j=1,\ldots,6$, let $e_j$ be the $j$--th unit basis vector of $\RR^6$. 
The vector obtained by projecting $e_j$ onto the complex tangent plane $V_p(U)$ is
\begin{align*}
e_j  - \frac{e_j\cdot \nabla f(p)}{\|\nabla f(p)\|^2}\cdot \nabla f(p)& - \frac{e_j\cdot J( \nabla f(p))}{\|J(\nabla f(p))\|^2}\cdot J( \nabla f(p)) \\
&= e_j  - \frac{e_j\cdot \nabla f(p)}{\|\nabla f(p)\|^2}\cdot \nabla f(p) - \frac{e_j\cdot J( \nabla f(p))}{\|\nabla f(p)\|^2}\cdot J( \nabla f(p)).
\end{align*}
Motivated by this observation, for each $p\in\RR^6$ and each index $j=1,\ldots,6$, define
$$
E_{j,f}(p) = e_j \cdot \|\nabla f(p)\|^2 - \big(e_j\cdot\nabla f(p)\big)\nabla f(p) - \big(e_j\cdot J(\nabla f(p))\big)J(\nabla f(p)).
$$
For each index $j$, if $\nabla f(p)$ is zero then $E_{j,f}(p)=0$.
Otherwise, the vector $E_{j,f}(p)$ has the direction obtained by projecting the $e_j$ onto the complex tangent plane $V_p(U)$. This vector is 0 when $e_j$ is orthogonal to $V_p(U)$.
Note that $E_{j,f}(p)\colon\RR^6\to\RR^6$ is a tuple of polynomials of degree at most $2(D-1)$.

Set $\mathbf{E}_{f}(p) = (E_{1,f}(p),\ldots, E_{6,f}(p))$.
Let $W\colon\RR^6\times\RR^6\to\RR$ be defined as
\begin{equation*}
W(p,v)=  f(p + v\cdot \mathbf{E}_{f}(p)).
\end{equation*}
For $p\in U$ define $W_p:\RR^6\to \RR$ as $W_p(v) = W(p,v)$. If $\nabla f(p) = 0$ then $W_p(v)$ is the zero polynomial. Otherwise, $W_p(v)$ is the zero polynomial if and only if the complex plane $V_p(U)$ is contained in $U$. 
Note that $W(p,v)$ has degree at most $D$ in the variables $v_1,\ldots,v_6$. Thus, we may write
$$
W_p(v)=\sum_{I} Q_I(p)v^I,
$$
where the sum is over all multi-indices $I = (j_1,\ldots,j_6)$ of weight at most $D$, and $Q_I(p)$ is a polynomial in $p$ of degree at most $2D(D-1)$. Define
$$
U^\prime = U\ \cap\ \bigcap_{I}\vb(Q_I).
$$
Then $U^\prime\subset\RR^6$ is the union of $U_{\operatorname{sing}}$ with the set of points $p\in U$ for which the complex tangent plane $V_p(U)$ is contained in $U$. If $U^\prime = U$ then $U$ is almost ruled by complex planes. 

If $U^\prime$ is a proper sub-variety of $U$, then it has dimension at most four and is defined by polynomials of degree at most $2D(D-1)$. By Lemma \ref{lem:comps}, $U$ contains $O(D^{12})$ complex planes. In the following paragraph we obtain a stronger bound. 

A generic complex line in $\RR^6$ intersects each complex plane in $U$ at a distinct point. 
We consider such a line $L$ that is not contained in $U$. 
Then $L\cap U(f)$ has dimension at most one. We identify $L$ with $\RR^2$ and define $\tilde f\in\RR[x,y]$ as the polynomial obtained by restricting $f$ to $L$. Similarly, there are $f_1,\ldots,f_k\in \RR[x,y]$ of degree at most $2D(D-1)$ such that the finite set $U^\prime\cap L$ corresponds to $\vb(\tilde f,f_1,\ldots,f_k)$. Without loss of generality, we can suppose that $\tilde f,f_1,\ldots,f_k$ do not share any common factors. If $k=0$ then Lemma \ref{lem:Harnack} implies $|U^\prime\cap L|\leq D$. If $k\geq 1$ then applying Lemma \ref{lem:Bezout} to $\tilde f$ and $f_1$ leads to $|U^\prime\cap L|\leq 2D^2(D-1)$. We conclude that $U$ contains at most $2D^2(D-1)$ complex planes. 
\end{proof}

\subsection{Point-line incidences inside a real hypersurface: rich surfaces}
\begin{lemma}\label{lem:eitherSSubsetHairLOrFewCurves}
Let $S\subset\RR^8$ be an irreducible surface defined by polynomials of degree at most $D$. Let $\pts\subset\RR^6$ and suppose $\dim(G_p\cap S)=1$ for all $p\in\pts$. Let $L\subset\RR^6$ be a complex line with $G(L)\in S$. Then either $S\subset \hair(L)$ or $|\pts\cap L|=O_D(1)$. 
\end{lemma}
\begin{proof}
For every $p\in\pts\cap L$ we have $G_p\subset H(L)$, and thus $G_p\cap S \subset \hair(L)\cap S$. If $p,q\in \pts$ are distinct then $G_p\cap G_q$ has cardinality at most one, so $|\pts\cap L|$ is bounded by the number of irreducible curves in $\hair(L)\cap S$. If $S\not\subset\hair(L)$, then by Lemma \ref{lem:comps} this quantity is $O_D(1)$. 
\end{proof}

\begin{lemma} \label{lem:incidenceBoundInsideS}
Let $U$ be a proper subvariety of $\RR^6$ defined by polynomials of degree at most $E$.  Let $S\subset\RR^8$ be an irreducible surface defined by polynomials of degree at most $D$. Let $\pts\subset U_{\reg}$ be a set of $m$ points. Let $\mathcal{L}$ be a set of $n$ complex lines contained in $U$ but not in any complex plane that is contained in $U$. Furthermore, suppose $G(L)\in S$ for each $L\in\mathcal{L}$. Then
\begin{equation}\label{eq:incidenceBoundInsideS}
I(\pts,\mathcal{L})=O_{D,E}(m+n).
\end{equation} 
\end{lemma}
\begin{proof}
If $p\in U_{\reg}$ is contained in a complex plane $\Pi\subset U$, then $\Pi = V_pU$ and any complex line $L\subset U$ that is incident to $p$ must be contained in $\Pi$. In particular we can assume that no points from $\pts$ are contained in a complex plane in $U$, since such points cannot contribute any incidences. 

For each $p\in\pts$, define $\gamma_p = G_p\cap G(U)$. If $\dim(\gamma_p\cap S) \le 0$ then Lemma \ref{lem:comps} implies $|\gamma_p\cap S|=O_{D,E}(1)$. The number of incidences formed by points of this type is $O_{D,E}(m)$. 

Set $\pts_1 =\{p\in\pts\colon \dim(\gamma_p\cap S)=1\}$. 
Assume that $S\not\subset G(U)$.
In this case, $|\pts_1|$ is at most the number of irreducible one-dimensional components of $S\cap G(U))$. Lemma \ref{lem:comps} implies that $|\pts_1|=O_{D,E}(1)$, which in turn leads to $I(\pts_1,\mathcal{L})=O_{D,E}(n)$. This establishes \eqref{eq:incidenceBoundInsideS}.

Next, assume that $S \subset G(U)$.  Consider two lines $L,L^\prime$ such that $G(L),G(L^\prime)\in S$ and $S\subset \hair(L)\cap \hair(L^\prime)$. Since $G(L^\prime)\in S\subset\hair(L)$, the lines $L$ and $L^\prime$ intersect at some point $p\in\RR^6$ and span a complex plane $\Pi$. This implies that $S\subset \hair(L)\cap \hair(L^\prime) = G_p\cap G(\Pi)$. Recalling that for every $p\neq q$ we have $|G_{p}\cap G_q|\leq 1$. Thus, $S\subset G_p,$ leads to $\pts_1\subseteq\{p\}$. In this case $I(\pts_1,\mathcal{L})\leq n$, which establishes \eqref{eq:incidenceBoundInsideS}. 

Finally, suppose that at most one line $L$ satisfies $G(L)\in S$ and $S\subset\hair(L)$. This line contributes at most $m$ incidences. By Lemma \ref{lem:eitherSSubsetHairLOrFewCurves}, every $L^\prime\in\mathcal{L}\setminus \{L\}$ is incident to $O_D(1)$ points from $\pts_1$. Once again, \eqref{eq:incidenceBoundInsideS} holds.
\end{proof}

\subsection{Point-line incidences inside a real hypersurface: preliminary bounds}

In this section we prove several preliminary bounds on the number of point-line incidences inside a real hypersurfaces. In the next section we use these results to prove Proposition \ref{prop:structureHypersurfaces}.

\begin{lemma} \label{le:RichInHyperFixedR}
Let $U\subset \RR^6$ be an irreducible variety defined by polynomials of degree at most $D$. 
Let $\lines$ be a set of $n$ complex lines that are contained in $U$ but are not in any complex plane contained in $U$. Let $r_0 = D^2+1$. Then
 
\[ |\pts_{r_0}(\lines)| = O_D(n^{3/2}). \]
\end{lemma}
\begin{proof}
Let $\mathcal{H}$ be the set of complex planes $H\subset\CC^3$ that contain at least $2n^{1/2}$ lines from $\lines$. By Lemma \ref{lem:FewRichSurfaces} we have $|\mathcal{H}|\leq n^{1/2}$. Since no line from $\lines$ is contained in a complex plane $H\subset U$, for each $H\in\mathcal{H}$ the intersection $H\cap U$ is a variety of dimension at most 3 defined by polynomials of degree at most $D$. By Lemma \ref{lem:exceptionalPtsInCPlane},
\begin{equation*}
|\pts_{r_0}(\lines_H)| = O_D(1).
\end{equation*}

Let $\mathcal{W}$ be the set of complex irreducible degree two surfaces $W\subset\CC^3$ that contain at least $8n^{1/2}$ lines from $\lines$. By Lemma \ref{lem:FewRichSurfaces} we have $|\mathcal{W}|\leq n^{1/2}$. Since the lines in an irreducible degree two surface determine at most one 3-rich point, for each $W\in\mathcal{H}$ we have
\begin{equation*}
|\pts_{r_0}(\lines_W)|\leq |\pts_3(\lines_W)| \leq 1.
\end{equation*}

If $L\in\lines$ is not contained in a plane from $\mathcal{H}$ then $L$ intersects each plane from $\mathcal{H}$ at most once. Similarly, if $L\in\lines$ is not contained in a surface from $\mathcal{W}$ then $L$ intersects each surface from $\mathcal{W}$ at most twice. Let $\lines^\prime$ be the set of lines that are not contained in a plane from $\mathcal{H}$ nor a surface from $\mathcal{W}$. Using Theorem \ref{th:Kollar} to control the contribution from $\lines^\prime$, we conclude that
\begin{equation*}
\begin{split}
|\pts_{r_0}(\lines)|&\leq |\pts_{2}(\lines^\prime)| + |\lines|(|\mathcal{H}|+2|\mathcal{W}|) + O_D(|\mathcal{H}|)+|\mathcal{W}|\\
&\leq O(n^{3/2}) + n(n^{1/2} + 2n^{1/2}) + O_D(n^{1/2}) + n^{1/2}\\
& = O_D(n^{3/2}).
\end{split}
\end{equation*}
\end{proof}

A routine random sampling argument allows us to bound the number of lines that are $r$ rich for larger values of $r$. For details, see for example \cite[Section 3]{Zahl20}. 
\begin{corollary}\label{cor:boundRRichPts}
Let $U\subset \RR^6$ be an irreducible variety defined by polynomials of degree at most $D$. 
Let $\lines$ be a set of $n$ complex lines that are contained in $U$ but are not in any complex plane contained in $U$. Then for each $r> D^2$,  
\[ |\pts_r(\lines)| = O_D(n^{3/2}r^{-3/2}). \]
\end{corollary}

\begin{corollary} \label{cor:DualBootstrapBound}
Let $U\subset \RR^6$ be an irreducible variety defined by polynomials of degree at most $D$.  
Let $\pts$ be a set of $m$ points in $\RR^6$. 
Let $\lines$ be a set of $n$ complex lines that are contained in $U$ but are not in any complex plane contained in $U$.
Then 
\[ I(\pts,\lines) = O_D\left(n^{3/2}+m\right). \]
\end{corollary}
\begin{proof}
Let $j_0$ be the smallest integer such that $2^j > D^2$. By Corollary \ref{cor:boundRRichPts},

\begin{equation*}
\begin{split}
I(\pts,\lines) &\leq 2(D^2+1) m + \sum_{j=j_0}^\infty 2^{j+1}|\pts_{2^j}(\lines)|\\
&\leq 2(D^2+1)m + O_D(n^{3/2}) \sum_{j=0}^\infty 2^j\cdot 2^{-3j/2}\\
& = O_D(n^{3/2}+m).
\end{split}
\end{equation*}
\end{proof}

\subsection{Proof of Proposition \ref{prop:structureHypersurfaces}}

We are now ready to prove Proposition \ref{prop:structureHypersurfaces}. The crucial step is to establish the following incidence result for points and curves in $\RR^3$. The following proof is closely modeled on the arguments in \cite{Zahl20} by the second author, which are in turn based on arguments of Sharir and Zlydenko \cite{SZ20}.

Before stating the next result, it will be helpful to introduce a definition. Let $\pts$ be a set of points and let $\curves$ be a set of curves in $\RR^3$. Let $K\colon\mathbb{N}\to\RR$ be a non-decreasing function. We say that $\pts$ and $\curves$ have $K$-\emph{good incidence geometry inside surfaces} if for every irreducible polynomial $f\in\RR[x,y,z]$, point set $\pts^\prime\subset\pts\cap \vb(f)$, and set $\curves^\prime\subset\curves$ of curves contained in $\vb(f)$, we have
\begin{equation*}
I(\pts,\curves^\prime)\leq K(\deg f)(|\mathcal{Q}^\prime|+|\curves^\prime|).
\end{equation*}
In brief, $\pts$ and $\curves$ have good incidence geometry inside surfaces if there do not exist large subsets of $\pts$ and $\curves$ that cluster into low degree surfaces and generate many incidences therein. 
We think of $K$ as an increasing function, so $K(t)$ is a valid constant when dealing with polynomials of degree at most $t$.

\begin{lemma} \label{le:PointCurveR3Inc}
Let $E,B\geq 1$ and let $K\colon\mathbb{N}\to\RR$ be a function. Let $\pts$ be a set of $n$ points in $\RR^3$ and let $\curves$ be a set of $m$ irreducible curves in $\RR^3$, each defined by polynomials of degree at most $E$. Suppose that $\pts$ and $\curves$ have $K$-good incidence geometry inside surfaces, and that for all set $\pts^\prime\subset\pts$ and $\curves^\prime\subset\curves$, we have 
\begin{equation} \label{eq:SZbootstrappingAssumption}
I(\pts^\prime,\curves^\prime) \leq B (|\pts^\prime|^{3/2}+|\curves^\prime|).
\end{equation}

Then
\begin{equation*}
I(\pts,\curves) \leq C(m^{3/5}n^{3/5}+m+n),
\end{equation*}
where the constant $C$ depends on $B,E,$ and $K(t)$ with $t = O_E(1)$. 
\end{lemma}
\begin{proof}
In what follows all implicit constants may depend on $B$ and $E$. We prove the result by induction on $m$. The base case of the induction is $m\leq m_0$, where $m_0$ is a constant specified below.
This base case holds by taking $C$ to be sufficiently large. 

Next, suppose that $m^2 \ge c n^{3}$, where $c$ is a constant specified below. In this case, \eqref{eq:SZbootstrappingAssumption} implies 
\[ I(\pts,\lines) = O\left(n^{3/2}+m\right) = O\left(m^{3/5}n^{3/5}+m\right),\]
where the implicit constant depends on $c$. If $C$ is sufficiently large compared to $c$ and $B$, then the induction closes and we are done. Henceforth we will assume that 
\begin{equation*} 
m_0<m<c^{1/2}n^{3/2}.
\end{equation*}

\parag{Partitioning the space.}
We refer to the case where $m\ge n^{2/3}$ as \emph{Case 1} and to the case where $m< n^{2/3}$ as \emph{Case 2}.
In Case 1 we set $D=\lfloor c n^{3/5}m^{-2/5}\rfloor$ and in Case 2 we set $D=\lfloor c m^{1/2}\rfloor$.
It can be easily verified that $D\le c n^{1/3}$ in both cases. 
At this point we fix $m_0 = \lceil c^{-2}\rceil$; this ensures that $D\geq 1$.  
To recap, in both Case 1 and Case 2 we have
\begin{equation*}
1 < D \le c n^{1/3}. 
\end{equation*} 
In addition, in both Case 1 and Case 2 we have 
\begin{equation*}
D\leq c m^{1/2}.
\end{equation*}

We apply Corollary \ref{co:partitionCombined} to obtain a nonzero polynomial $f\in \RR[x_1,x_2,x_3]$ of degree at most $D$ that satisfies the following. 
Each connected component of $\RR^3\setminus \vb(f)$ intersects $O(mD^{-2})$ varieties from $\curves$ and contains $O(nD^{-3})$ points from $\pts$. As discussed in Remark \ref{rem:coDimOne}, we may suppose that $\vb(f)$ is a surface (that is, $\vb(f)$ is equidimensional and each irreducible component has dimension two). 

Recall that $\RR^3\backslash \vb(f)$ is a union of $O(D^3)$ cells. For each index $j=1,\ldots,O(D^3)$, let $\pts_j$ be the set of points of $\pts$ in the $j$-th cell and let $\curves_j$ be the set of elements of $\curves$ that intersect the $j$-th cell.
We also set $\pts_0 = \pts\cap \vb(f)$, $n_0 = |\pts_0|$, and $n' = n-n_0$.
By definition, for each index $j$ we have $|\pts_j|=O(nD^{-3})$ and $|\curves_j|=O(mD^{-2})$.
Note that $n' = \sum_j |\pts_j|$.

We first bound $I(\pts\setminus \pts_0,\curves)$.
In Case 1, by applying \eqref{eq:SZbootstrappingAssumption} separately in each cell, we obtain
\begin{equation}\label{eq:cellCase1}
\begin{split}
I(\pts\setminus \pts_0,\curves) & = \sum_j O\big(|\pts_j|^{3/2}+|\curves_j|\big) = O\big(D^3\cdot \big(\big(\frac{n}{D^3}\big)^{3/2}+\frac{m}{D^2}\big)\big)\\
&= O\big(\frac{n^{3/2}}{D^{3/2}}+m\cdot D\big)  =O\big(m^{3/5}n^{3/5}\big). 
\end{split}
\end{equation}

In Case 2, the number of elements of $\curves$ that intersect a cell is $O(mD^{-2}) = O(1)$. Since each such cell contains $O(nD^{-3})$ points, we obtain
\begin{equation} \label{eq:cellCase2} 
I(\pts\setminus \pts_0,\curves) \le \sum_j |\pts_j||\curves_j| = \sum_j O\left(|\pts_j|\right) = O(n').
\end{equation}

\parag{Handling curves on the partition.}
It remains to derive an upper bound on $I(\pts_0,\curves)$. As discussed above, we can assume that each irreducible component of $\vb(f)$ has dimension two. Denote these two-dimensional components as $U_1,\ldots,U_{D^\prime}$, where $D^\prime\leq D$.
For each $1\le j \le D'$, let $f_j$ be a minimum degree polynomial satisfying $\vb(f_j)=U_j$.

We set $I(\pts_0,\curves) = I' + I''$ as follows.
Let $p\in \pts_0$ be incident to $\gamma \in \curves$. If there is an index $1\leq j\leq D^\prime$ so that $p\in U_j$ and $\dim(U_j\cap\gamma)=0$, then the incidence $(p,\gamma)$ contributes to $I'$. Otherwise $(p,\gamma)$ contributes to $I''$. 

Since each element of $\curves$ is defined by polynomials of degree at most $E$, if $\gamma\in\curves$ is not contained in $U_j$ then Lemma \ref{th:VarietyPartComponents} implies $|\gamma\cap U_j|=O(D_j)$. In particular, each curve $\gamma\in\Gamma$ can contribute $O(D)$ incidences of the form $I'$. Thus 
\begin{equation}\label{eq:IPrimeBd}
I'=O(Dm)=O(m^{3/5}n^{3/5}).
\end{equation} 

It remains to bound $I''$. Let $\pts'_j$ denote the set of points $p$ of $\pts_0$ such that $p\in U_j$ and $p\not\in U_{j'}$ for each $j'<j$.
Let $\curves'_j$ denote the set of curves $\gamma\in\curves$ for which $\gamma\subset U_j$ and $\gamma\not\subset U_{j'}$ for each $j'<j$. For every incidence $(p,\gamma)$ contributing to $I''$ there is an index $j$ such that $p\in\pts'_j$ and $\gamma\in\curves'_j$.

For each index $j$, applying Lemma \ref{lem:tooManyExceptionsImpliesDoublyRuled} to $U_j$ implies the following. Either $U_j$ contains $O(D^2)$ exceptional curves defined by polynomials of degree $O(1)$, or $U_j$ is doubly ruled by such curves. In the latter case $\deg(f_j) = O(1)$, where the implicit constant depends only on $E$. By re-indexing, we can suppose that $U_1,\ldots,U_{h}$ are doubly ruled by curves defined by polynomials of degree $O(1)$, and $U_{h+1},\ldots,U_{D^\prime}$ are not. (If no $U_i$ is doubly ruled then we set $h = 0$. If all $U_i$ are doubly ruled then we set $h = D^\prime$.) 

Since $\pts$ and $\curves$ have $K$-good incidence geometry inside surfaces, for each index $j=1,\ldots,h$ we have
$$
I(\pts'_j,\curves'_j)= O(|\pts'_j|+|\curves'_j|),
$$
where the implicit constant only depends on $K(t)$ with $t = O_E(1)$. Thus 
\begin{equation}\label{eq:IncidencesIppDoublyRuled}
\sum_{j=1}^h I(\pts'_j,\curves'_j)=O\Big(\sum_{j=1}^h |\pts'_j| + \sum_{j=1}^h |\curves'_j|\Big).
\end{equation}

It remains to control incidences $(p,\gamma)$ where $p\in\pts'_j$ and $\gamma\in\curves'_j$ for some $h+1\leq j\leq D'$. We call a point $p\in\pts'_j$ \emph{rich} if it is incident to at least two curves from $\curves'_j$. Otherwise $p$ is \emph{poor}. For each index $j=h+1,\ldots D'$, let $\pts'_{j,\operatorname{rich}}$ and $\pts'_{j,\operatorname{poor}}$ be the set of rich and poor points of $\pts'_j$, respectively. Define
$$
\pts_{\operatorname{rich}}=\bigcup_{j=h+1}^{D'} \pts'_{j,\operatorname{rich}}\quad \text{ and } \quad
\pts_{\operatorname{poor}}=\bigcup_{j=h+1}^{D'} \pts'_{j,\operatorname{poor}}.
$$

Set $n_{\operatorname{poor}} = |\pts_{\operatorname{poor}}|$ and $n_{\operatorname{rich}} = |\pts_{\operatorname{rich}}|$.
Note that $n_{\operatorname{poor}}+n_{\operatorname{rich}} \le n_0$. We have
\begin{equation}\label{eq:poorPointIncidences}
\sum_{j=h+1}^{D'} I(\pts'_{j,\operatorname{poor}},\curves'_j)\leq 2\sum_{j=h+1}^{D'}|\pts'_{j,\operatorname{poor}}|=2n_{\operatorname{poor}}. 
\end{equation}

In a similar vein, let $\Gamma'_{j,\operatorname{exceptional}}$ be the set of exceptional curves in $\Gamma_j'$, and let $\Gamma'_{j,\operatorname{plebeian}}$ be the set of non-exceptional curves. Define
$$
\curves_{\operatorname{exceptional}}=\bigcup_{j=h+1}^{D'} \curves'_{j,\operatorname{exceptional}}\quad \text{ and }\quad 
\curves_{\operatorname{plebeian}}=\bigcup_{j=h+1}^{D'} \curves'_{j,\operatorname{plebeian}}.
$$
Since each curve $\gamma\in \curves'_{j,\operatorname{plebeian}}$ is incident to $O(1)$ rich points, we have
\begin{equation}\label{eq:plebCurveIncidences}
\sum_{j=h+1}^{D'} I(\pts'_{j,\operatorname{rich}},\curves'_{j,\operatorname{plebeian}})= \sum_{j=h+1}^{D'} O(|\curves'_{j,\operatorname{plebeian}}|).
\end{equation}

Finally, Lemma \ref{lem:tooManyExceptionsImpliesDoublyRuled} implies
$$
|\curves_{\operatorname{exceptional}}| = \sum_{j=h+1}^{D'}O_E( (\deg f_j)^2) = O_E(D^2) = O_E(c^2m).
$$
If the constant $c$ is selected sufficiently small compared to $E$, then
$$
|\curves_{\operatorname{exceptional}}|  \leq m/2.
$$
We can now apply the induction hypothesis to conclude that
\begin{equation}\label{eq:inductionBd}
\begin{split}
I(\pts_{\operatorname{rich}}, \curves_{\operatorname{exceptional}}) & \leq C(|\pts_{\operatorname{rich}}| |\curves_{\operatorname{exceptional}}|^{3/5} + |\pts_{\operatorname{rich}}| +  |\curves_{\operatorname{exceptional}}|)\\
&\leq C(m^{3/5}n^{3/5}/2^{3/5} + n_{\operatorname{rich}} + m/2).
\end{split}
\end{equation}

\parag{Wrapping up.}
In the above, we partitioned the incidences of $\pts\times \curves$ into several cases.
The incidences inside the cells of the partition are bounded in \eqref{eq:cellCase1} and \eqref{eq:cellCase2}.
The number of incidences with points on the variety of the partition was split into $I'$ and $I''$. In \eqref{eq:IPrimeBd} we bounded $I'$. The incidences of $I''$ were further partitioned and bounded in \eqref{eq:IncidencesIppDoublyRuled}, \eqref{eq:poorPointIncidences}, \eqref{eq:plebCurveIncidences}, and \eqref{eq:inductionBd}.
Combining all these bounds gives
\[ I(\pts,\lines) = O\left(m^{3/5}n^{3/5}+m+n'\right) + 2n_{\operatorname{poor}}+ C(m^{3/5}n^{3/5}/2^{3/5} + n_{\operatorname{rich}} + m/2), \]
where the implicit constant depends on $B,E,$ and $K(t)$ with $t = O_E(1)$. 
Recall that $n_\text{rich}+ n_\text{poor} + n' = n$. 
By taking $C$ to be sufficiently large with respect to $B,E,$ and $K(t)$, we obtain 
\[ I(\pts,\lines) \le C\left(m^{3/5}n^{3/5}+m+n\right). \]
This closes the induction and finishes the proof.
\end{proof}

\begin{lemma}\label{lem:incidencesPtsCplxPlanes}
Let $U\subset \RR^6$ be an irreducible variety defined by polynomials of degree at most $E$ that is not almost ruled by complex planes. 
Let $\pts\subset U_{\reg}$ be a set of $m$ points.
Let $\lines \subset \RR^6$ be a set of $n$ complex lines that are contained in $U$ but not in any complex plane contained in $U$.
Then 
\[ I(\pts,\lines) \leq C(m^{3/5}n^{3/5}+m+n). \]
\end{lemma}
\begin{proof}
First, we can assume that for each point $p\in\pts$, the set $G(U)\cap G_p$ has dimension at most one. Indeed, if $G(U)\cap G_p$ has dimension two, then by Lemma \ref{lem:DenseInPlane}, the complex plane $V_pU$ is contained in $U$. If $L$ is a complex line contained in $U$ that is incident to $p$, then $L\subset V_pU$. By hypothesis, such a line $L$ cannot be an element of $\lines$. Thus we can discard any point $p\in\pts$ where $G(U)\cap G_p$ has dimension $\geq 2$, since such points cannot contribute any incidences. 

For each $p\in\pts$, define $\beta_p = G(U)\cap G_p$; this is a variety in $\RR^8$ of dimension at most one. We have 
$$
I(\pts,\lines)=I(\{G(L)\colon L\in\lines\},\{\beta_p\colon p\in\pts\}\}.
$$ 

Let $\mathcal{Y} = \{G(L)\colon L\in\lines\}\cup \{\beta_p\colon p\in\pts\}$ and let 
$$
\mathcal{X} = \big\{\bigcup_{Y\in \mathcal{Y}^\prime} Y\colon \mathcal{Y}^\prime\subset \mathcal{Y}\big\}.
$$
That is, $\mathcal{X}$ is the family consisting of all finite unions of points from $\{G(L)\colon L\in\lines\}$ and curves from $\{\beta_p\colon p\in\pts\}$. 

Use Lemma \ref{lem:goodProjectionSet} to select a projection $\pi\colon\RR^8\to\RR^3$ that is non-degenerate with respect to $\mathcal{X}$. Let $\mathcal{Q}=\{\pi(G(L))\colon L\in\lines\}$ and let $\curves$ be the set of all irreducible curves $\gamma\subset\RR^3$ with $\gamma\subset \overline{\pi(\beta_p)}$ for some $p\in\pts$. By Lemma \ref{lem:comps}, $|\curves|=O(n)$, and each curve $\gamma\in\curves$ is defined by polynomials of degree $O(1)$. 

We claim that
\begin{equation}\label{eq:incidencesPtsCurvesControlsPtsLines}
I(\pts,\lines)\leq I(\mathcal{Q},\curves)+O(m).
\end{equation}
Indeed, if $p\in\pts$ is incident to $L\in\lines$, then either there is an irreducible curve $\gamma\subset \overline{\pi(\beta_p)}$ with $G(L)\in\gamma$, or $G(L)$ is a zero-dimensional component of $\beta_p$. By Lemma \ref{lem:comps}, $\beta_p$ has $O(1)$ irreducible components, so there are $O(m)$ incidences of this type. 

Since the projection $\pi$ non-degenerate, it does not introduce new incidences. By Corollary \ref{cor:DualBootstrapBound}, there is a constant $B$ depending on $E$ such that for all sets $\mathcal{Q}^\prime\subset\mathcal{Q}$ and $\curves^\prime\subset\curves$,
\begin{equation*} 
I(\mathcal{Q},\curves^\prime) \leq B(|\mathcal{Q}^\prime|^{3/2}+|\curves^\prime|).
\end{equation*}

Again, since $\pi$ is non-degenerate, Lemma \ref{lem:incidenceBoundInsideS} implies that $\mathcal{Q}$ and $\Gamma$ have $K$-good incidence geometry inside surfaces, where $K\colon\mathbb{N}\to\RR$ is a function that depends only on $E$. Applying Lemma \ref{le:PointCurveR3Inc} to $\mathcal{Q}$ and $\curves$, we conclude that
\begin{equation}\label{eq:BdOnIQGamma}
I(\mathcal{Q},\curves) = O(m^{3/5}n^{3/5}+m+n).
\end{equation}
The result now follows by combining \eqref{eq:incidencesPtsCurvesControlsPtsLines} and \eqref{eq:BdOnIQGamma}.
\end{proof}

Using Lemma \ref{lem:incidencesPtsCplxPlanes}, we can now prove Proposition \ref{prop:structureHypersurfaces}. We first recall the statement of this proposition.

\begin{propositionStructureHypersurfaces}
Let $U\subset\RR^6$ be an irreducible variety defined by polynomials of degree at most $D$. Then at least one of the following two statements holds.
\begin{itemize}[noitemsep]
\item $U$ is almost ruled by complex planes. 

\item $U$ contains at most $2D^2(D-1)$ complex planes. If $\lines$ is a set of $n$ complex lines that are contained in $U$ but not contained in any of these planes, then for each $r> D^2$ we have
\begin{equation*}
|U_{\operatorname{reg}}\cap \pts_r(\lines)| = O_D(n^{3/2}r^{-5/2}+nr^{-1}).
\end{equation*}
\end{itemize}
\end{propositionStructureHypersurfaces} 
\begin{proof}
If $U$ is almost ruled by complex planes then the first item holds and we are done. Suppose that $U$ is not almost ruled by complex planes. Then by Lemma \ref{lem:varietyRuledPlanesOrFewPlanes}, $U$ contains at most $2D^2(D-1)$ complex planes. Let $\lines$ be a set of complex lines contained in $U$ that are not contained in any complex plane contained in $U$. Let $r>D^2$ and let $C=C(D)$ be the constant from Lemma \ref{lem:incidencesPtsCplxPlanes}. 

We first consider the case of $r<2C$. 
Let $\surfs$ be the set of complex planes that contain at least $2n^{1/2}$ lines from $\lines$. We also insert into $\surfs$ every irreducible degree two surface that contains at least $8n^{1/2}$ lines from $\lines$. Lemma \ref{lem:FewRichSurfaces} implies $|\surfs|= O(n^{1/2})$. By Lemma  \ref{lem:exceptionalPtsInCPlane}, for each plane $S\in\surfs$ the number of $r$-rich points formed by lines in $\lines_S$ is $O_D(1)$. The lines contained in an irreducible degree two surface form at most one 3-rich point. The total contribution from all elements in $\surfs$ is $O_D(n^{1/2})$. A line $L\in \lines$ intersects a surface of $\surfs$ that does not contain $L$ in at most two points. Summing this over all lines of $\lines$ and all surfaces of $\surfs$ leads to $O(n^{3/2})$ intersection points. Thus, removing from $\lines$ all lines that are contained in at least one surface of $\surfs$ decreases the number of $r$-rich points by $O(n^{3/2})$. 

After the above pruning of $\lines$, we can apply Theorem \ref{th:Kollar} on $\lines$. The theorem states that $\pts_{r}(\lines) = O_D(n^{3/2})$. 
By the above, after bringing back the removed lines the number of $r$-rich points remains $O_D(n^{3/2})$.
Since $r<2C$, we have that $O_D(n^{3/2}) = O_D(n^{3/2}/r^{5/2})$.

We move to consider the case of $r\geq 2C$. Set $\pts = U_{\reg}\cap \pts_r(\lines)$. Applying Lemma \ref{lem:incidencesPtsCplxPlanes}, we have
\[ r\cdot |\pts| \le I(\pts,\lines) \leq C\left(|\pts|^{3/5}n^{3/5}+|\pts|+n\right). \]
Since $r\ge 2C$, after subtracting $C|\pts|$ from both sides we obtain
\[ r\cdot |\pts|/2 \leq C\left(|\pts|^{3/5}n^{3/5}+n\right). \]
Rearranging yields the bound in the statement of the proposition. 
\end{proof}

\section{A structure theorem for lines in $\CC^3$} \label{sec:LineIncC3}
In this section we prove Theorem \ref{th:RichC}. As discussed in the introduction, the theorem is proved by induction on the number of lines. We use Theorem \ref{th:partition} to partition $\RR^6$ into open connected cells, and apply the induction hypothesis inside each cell. The main difficulty occurs when many of the $r$-rich points are contained in the boundary $Z(f)$ of the partition. If an $r$-rich point $p$ is contained in $Z(f)$, then either many complex lines incident to $p$ are contained in $Z(f)$, or many such lines properly intersect $Z(f)$. In Section \ref{ssec:complexContainedInReal} we develop tools to understand the former case, and in Section \ref{ssec:complexIntersectReal} we develop tools to understand the latter. Finally, in Section \ref{ssec:proofOfThemRIchC} we use these tools to prove Theorem \ref{th:RichC}.

\subsection{A structure theorem for complex lines inside a real variety}\label{ssec:complexContainedInReal}

\begin{lemma}\label{lem:incidencesInsideU}
Let $U\subset\RR^6$ be an irreducible variety defined by polynomials of degree at most $D$. Let $\mathcal{L}$ be a set of $n$ complex lines that are contained in $U$. Let $r$ be sufficiently large compared to $D$. Then there is a set $\surfs$ of complex planes in $\CC^3$ such that each $S\in\surfs$ contains at least $2rn^{1/2+\eps}$ lines from $\lines$, and 
\begin{equation}\label{eq:complexLinesRRich}
\Big|(U_{\reg}\cap\pts_r(\lines))\backslash \bigcup_{S\in\surfs}\pts_r(\mathcal{L}_{S})\Big|= O_D( n^{3/2+\eps}r^{-2} + nr^{-1}).
\end{equation}
\end{lemma}
\begin{proof}
We first consider the case where $U$ is almost ruled by complex planes. 
In this case, for each $p\in U_{\reg}\cap\pts_r(\lines)$, every line from $\lines$ that contains $p$ is contained in $V_pU$. Let $\surfs$ be the set of planes contained in $U$ that contain at least $2rn^{1/2+\eps}$ lines from $\lines$. Let $\surfs^\prime$ be the set of planes contained in $U$ that contain between $1$ and $2rn^{1/2+\eps}$ lines from $\lines$. By Theorem \ref{th:complexST}, we have
\begin{equation*}
\begin{split}
\sum_{S\in\surfs^\prime}\pts_r(\mathcal{L}_S) & = \sum_{S\in\surfs^\prime}O( |\mathcal{L}_S|^2 r^{-3} + |\mathcal{L}_S|r^{-1})\\
&= O( rn^{1/2+\eps} \cdot \sum_{S\in\surfs^\prime}|\mathcal{L}_S| r^{-3} + \sum_{S\in\surfs^\prime}|\mathcal{L}_S|r^{-1}) =O(n^{3/2+\eps}r^{-2}+nr^{-1}).
\end{split}
\end{equation*}
We conclude that
\begin{equation*}
\begin{split}
\Big|(U_{\reg}\cap\pts_r(\lines))\backslash \bigcup_{S\in\surfs}\pts_r(\mathcal{L}_{S})\Big| & \leq \sum_{S\in\surfs^\prime}\pts_r(\mathcal{L}_S)\\
& = O(n^{3/2+\eps}r^{-2}+nr^{-1}).
\end{split}
\end{equation*}

We now consider the case where $U$ is not almost ruled by complex planes. 
By Proposition \ref{prop:structureHypersurfaces}, the variety $U$ contains at most $2D^2(D-1)$ complex planes. Let $\surfs_0$ be the set of complex planes in $U$. 
Let $\surfs$ be the set of planes in $\surfs_0$ that contain at least $2n^{1/2+\eps}r$ lines from $\lines$.

Consider $p\in U_{\reg}$ that is contained in a plane $S\in \surfs_0$. 
Then every line from $\lines$ that is incident to $p$ is contained in $\lines_S$. 
By repeating the above argument that involves Theorem \ref{th:complexST}, we have
\begin{equation}\label{eq:rRichInPoorSurfaces}
\sum_{S\in\surfs_0\backslash\surfs}|\pts_r(\lines_S)|= O(n^{3/2+\eps}r^{-2}+nr^{-1}).
\end{equation}

Let $\lines^\prime=\lines\backslash\bigcup_{S\in\surfs_0}\lines_S$.
By Proposition \ref{prop:structureHypersurfaces} we have
\begin{equation}\label{eq:rRichOutsideSurfaces}
|\pts_r(\lines^\prime)| = O(n^{3/2}r^{-5/2}+nr^{-1}).
\end{equation}
Combining \eqref{eq:rRichInPoorSurfaces} and \eqref{eq:rRichOutsideSurfaces} yields \eqref{eq:complexLinesRRich} and finishes the proof.
\end{proof}

\subsection{A structure theorem for complex lines intersecting a real variety}\label{ssec:complexIntersectReal}
In this section we analyze the structure of complex lines that properly intersect a real variety $U\subset\RR^6$ and determine many $r$-rich points therein. The basic idea is as follows. Suppose that $\lines$ is a set of complex lines in $\RR^6$ that properly intersect $U$. For each $L\in\mathcal{L}$, the intersection $L\cap U$ is a union of isolated points and real curves in $\RR^6$. Ignoring the isolated points, we can use Lemma \ref{th:RichR} to obtain a structure theorem for the set of curves $\{L\cap U\colon L\in\lines\}$. Lemma \ref{th:RichR} gives us a collection of irreducible real surfaces in $\RR^6$ that cover most of the $r$-rich points inside $U$. However, the surfaces $\surfs$ from Theorem \ref{th:RichC} are not irreducible real surfaces---they are complex surfaces of degree at most two. Lemma \ref{le:RealInComp} shows that for every irreducible real surface $S\subset\RR^6$ that contains many real curves of the form $L\cap U$, there is a complex surface in $\CC^3$ that contains the corresponding lines from $\lines$. 

Before getting to Lemma \ref{le:RealInComp}, we introduce terminology and results concerning the interplay between real and complex varieties. 
We now identify $\RR^{d}$ with the real part of $\CC^{d}$. Concretely, if $p=(p_1,\ldots,p_d)\in \RR^d$, then we define $\iota(p) = (p_1,\ldots,p_d)\in\CC^d$. If $U\subset\RR^d$ is a variety, we define the \emph{complexification} $U^*$ of $U$ to be the smallest variety in $\CC^d$ that contains $U$. This set is precisely the Zariski closure of $\iota(U)$. We have that $\dim_{\CC} U^* = \dim_{\RR} U$. In the opposite direction, if $W\subset\CC^d$ is a variety then we define $W(\RR)=\{ p \in \RR^d\colon \iota(p)\in W\}$ to be the set of real points of $W$. We have $\dim_{\RR} W(\RR)\leq \dim_{\CC}(W),$ and strict inequality is possible.
Further information can be found in \cite{Whitney57}.

If $U\subset \RR^d$ is a variety defined by polynomials of degree at most $D$, then $U^*\subset \CC^d$ is a variety of degree $O_{D,d}(1)$. Similarly, if $W\subset \CC^d$ is a variety of degree $D$, then  $W(\RR)\subset\RR^d$ is defined by polynomials of degree $O_{D,d}(1)$.

As in the previous sections, we also identify $\CC^d$ with $\RR^{2d}$ using the map
\begin{equation*}
(x_1+iy_1,\ldots,x_d+iy_d) \to (x_1,y_1,\ldots,x_d,y_d),
\end{equation*}
where $x_1,y_1,\ldots,x_d,y_d\in \RR$.

\begin{lemma} \label{le:RealInComp}
Let $U\subseteq\RR^{2d}$ be an irreducible variety defined by polynomials of degree at most $D$.
Then there exists an irreducible variety $W\subseteq\CC^d$ defined by polynomials of degree $O_{D,d}(1)$ such that $\dim_{\CC} W \le \dim_{\RR} U$ and $\iota(U) \subseteq W$.
\end{lemma}
\begin{proof}
We define the linear functions $\tau:\CC^{2d}\to \CC^{2d}$ as 
\begin{align*}
\tau(x_1,y_1,\ldots,x_d,y_d) = (x_1+iy_1, x_2+iy_2,\ldots, x_d+iy_d, x_1-iy_1, x_2-iy_2,\ldots, x_d-iy_d).
\end{align*}

Since the expressions $x_1+iy_1,\ldots, x_d+iy_d, x_1-iy_1,\ldots, x_d-iy_d$ are linearly independent, we may use them as another coordinate system.
Using these coordinates, we consider the projection $\pi:\CC^{2d}\to\CC^d$ defined by
\[ \pi(x_1+iy_1,\ldots, x_d+iy_d, x_1-iy_1,\ldots, x_d-iy_d) = (x_1+iy_1,\ldots, x_d+iy_d). \]

Recall that the complexification $U^*$ is defined by polynomials of degree $O_{D,d}(1)$.
By Lemma \ref{le:projDeg}, there exists a variety $W\subset\CC^n$ defined by polynomials of degree $O_{D,d}(1)$ such that $\pi(\tau(U^*))\subset W$.
It remains to show that $\iota(U) \subseteq W$.
Consider a point $p\in U$ and write
\[ p=(p_1,q_1,\ldots,p_d,q_d) \in \RR^{2d}. \]

Let $p^*$ be the point in $\CC^{2d}$ having the same coordinates as $p$.
Note that
\[ \tau(p^*)=(p_1+iq_1,\ldots,p_d+iq_d,p_1-iq_1,\ldots,p_d-iq_d) \in \CC^{2d}. \]
This implies that
\[ \pi_d(\tau(p^*)) = (p_1+iq_1,\dots,p_d+iq_d) \in \CC^d, \]
so
\[ \pi_d(\tau(p^*)) = \iota(p_1,\ldots,p_d,q_1,\ldots,q_d). \]

Since $\pi(\tau(p^*))=\iota(p)$, we have that $\iota(p)\in W$.
That is, $\iota(U) \subset W$. 
Since $U$ is irreducible, every component of $W$ either contains $\iota(U)$ or intersects $\iota(U)$ in a lower-dimension variety. 
Thus, at least one irreducible component of $W$ contains $\iota(U)$.
To complete the proof, select a component of $W$ that contains $\iota(U)$.
\end{proof}

\begin{lemma}\label{lem:richLinesCurveIntersect}
Let $U\subset\RR^6$ be a variety that is defined by polynomials of degree at most $D$. Let $\mathcal{L}$ be a set of complex lines that are not contained in $U$. Let $r$ be sufficiently large compared to $D$. Then there is a set $\surfs$ of complex planes in $\CC^3$ such that each $S\in\surfs$ contains at least $2rn^{1/2+\eps}$ lines from $\lines$ and 
\begin{equation*}
\Big|U\cap\pts_r(\lines)\backslash \bigcup_{S\in\surfs}\pts_{\frac{4}{5}r}(\mathcal{L}_{S})\Big|= O_{D,\eps}( n^{3/2+\eps}r^{-2}).
\end{equation*}
\end{lemma}
\begin{proof}
Define $\curves$ to be the set of all irreducible curves $\gamma\subset\RR^6$ such that $\gamma\subset L\cap U$ for some $L\in\mathcal{L}$. By Lemma \ref{lem:comps}, $|\curves|=O_D(n)$, and each curve in $\curves$ is defined by polynomials of degree $O_D(1)$. 

Let $\pts_0$ be the set of points incident to at least $r/100$ lines from $L$ at an isolated point of $U\cap L$.
Using Lemma \ref{lem:comps} again, for each $L\in\lines$, the intersection $L\cap U$ has $O_D(1)$ isolated points. This implies that $|\pts_0|=O(nr^{-1})$. Note that 
\begin{equation*}
U\cap\pts_r(\lines)\subset \pts_{\frac{99}{100}r}(\Gamma)\cup \pts_0.
\end{equation*}

We apply Lemma \ref{th:RichR} to $\Gamma$ with $\frac{99}{100}r$ in place of $r$, and let $\surfs^\prime$ be the resulting set of irreducible real surfaces of degree $O_{D,\eps}(1)$ in $\RR^6$. 
We have
\begin{equation*}
\Big|\pts_{\frac{99}{100}r}(\Gamma)\backslash \bigcup_{S^\prime\in\surfs^\prime}\pts_{\frac{4}{5}r}(\Gamma_{S^\prime})\Big|=O_{D,\eps}(n^{3/2+\eps}r^{-2}).
\end{equation*}

Apply Lemma \ref{le:RealInComp} to each surface $S\in\surfs^\prime$, to obtain an irreducible complex variety of dimension at most two. Let $\surfs^{\prime\prime}$ be the set of the resulting irreducible complex varieties. Note that $|\surfs^{\prime\prime}| = O_{D,\eps}(n^{1/2-\eps}r^{-1})$. Each variety of $\surfs^{\prime\prime}$ has an infinite intersection with a complex line and is thus two-dimensional.  
A surface of $\surfs^{\prime\prime}$ contains each complex lines that it has an infinite intersection with.
Thus, when $p\in\pts_{\frac{4}{5}r}(\Gamma_{S^\prime})$ for some $S^\prime\in\surfs^\prime$, there is a complex surface $S^{\prime\prime}\in\surfs^{\prime\prime}$ such that $p\in\pts_{\frac{4}{5}r}(\lines_{S^\prime})$. 
This implies 
\begin{equation*}
\Big|U\cap\pts_r(\lines)\backslash \bigcup_{S\in\surfs^{\prime\prime}}\pts_{\frac{4}{5}r}(\mathcal{L}_{r^\prime})\Big|= O_{D,\eps}( n^{3/2+\eps}r^{-2}).
\end{equation*}

Let $\surfs$ be the set of surfaces $S\in\surfs^{\prime\prime}$ that are complex planes containing at least $2n^{1/2+\eps}r$ lines from $\lines$. We claim that
\begin{equation}\label{eq:wonkySurfacesArentImportant}
\sum_{S\in\surfs^{\prime\prime}\backslash\surfs}|\pts_{\frac{4}{5}r}(\mathcal{L}_{S})|=O_{D,\eps}( n^{3/2+\eps}r^{-2}+nr^{-1}).
\end{equation}
Indeed, if $S\in \surfs^{\prime\prime}$ is a plane containing fewer than $2n^{1/2+\eps}r$ lines, then by Theorem \ref{th:complexST}, 
$$
|\pts_{\frac{4}{5}r}(\mathcal{L}_{S})|=O(n^{1+2\eps}r^{-1})
$$
Since there are $O(n^{1/2-\eps}r^{-1})$ such planes, their total contribution is $O(n^{3/2+\eps}r^{-2})$. 

Consider $S\in \surfs^{\prime\prime}$ that is not a plane.
Since $S$ contains at least $2n^{1/2+\eps}r$ lines, it must be a ruled surface. 
The lines in a ruled surface that is not a plane form at most one 3-rich point. (For these claims about ruled surfaces in $\CC^3$, see for example \cite{Kollar15}. In particular, see the part titled ``Special ruled surfaces''.)   
The total contribution from surfaces of this type is $O(n^{1/2-\eps}r^{-1})=O(n^{3/2}r^{-2})$. This establishes \eqref{eq:wonkySurfacesArentImportant}, which in turn completes the proof of the lemma. 
\end{proof}

\subsection{Proof of Theorem \ref{th:RichC}}\label{ssec:proofOfThemRIchC}
Armed with Lemmas \ref{lem:incidencesInsideU} and \ref{lem:richLinesCurveIntersect}, we are now ready to prove Theorem \ref{th:RichC}. For the reader's convenience we first recall the statement of the theorem. 

\begin{thmRichC}
For every $\eps>0$, there is a constant $C$ such that the following holds. Let $\lines$ be a set of $n$ lines in $\CC^3$, let $2\leq r\leq 2n^{1/2}$ and let $r^\prime = \max(2,  r/3)$. Then there exists a set $\surfs$ of algebraic surfaces in $\CC^3$ with the following properties.
\begin{itemize}[noitemsep]
\item If $r\geq 3$ then every surface in $\surfs$ is a plane. If $r=2$ then every surface in $\surfs$ is irreducible and has degree at most two.
\item Every plane $W\in \surfs$ contains at least $rn^{1/2+\eps}$ lines of $\lines$. 
\item $|\surfs|\le 2n^{1/2-\eps}r^{-1}$.
\item $|\pts_r(\lines)\setminus \bigcup_{W\in \surfs}\pts_{r^\prime}(\lines_W)|\leq C n^{3/2+\eps}r^{-2}$.
\end{itemize}
\end{thmRichC}
\begin{proof}
Recall that when $r=2$ or $r=3$, the result was proved in Section \ref{sec:KollarBd}. By taking $C$ to be sufficiently large and using the bound from the $r=3$ case, we obtain the result for any constant $r$. We may thus assume that $r\geq r_{\eps}$ for a sufficiently large $r_{\eps}$ depending on $\eps$. In particular, we can assume that $r^\prime = r/3$. 

With $\eps$ and $r$ fixed, we prove the result by induction on $n$. By selecting the constant $C$ sufficiently large, we can suppose that $n\geq n_{\eps}$ for a fixed value $n_{\eps}$ of our choosing. Suppose now that the result has been proved for all sets of lines of cardinality smaller than $n$, and let $\lines$ be a set of complex lines of cardinality $n$.

Since each complex line $L\in\lines$ is also a two-dimensional real variety in $\RR^6$, we can apply Theorem \ref{th:partition} to $\lines$ with a value of $D = D(\eps)$ to be specified below. We obtain a polynomial $f\in \RR[x_1,\ldots,x_6]$ of degree at most $D$ such that each connected component of $\RR^6\setminus \vb(f)$ intersects $O(nD^{-4})$ complex lines from $\lines$. We denote these open connected components as $\Omega_1,\ldots,\Omega_s,$ with $s = O(D^6)$. Let $U = \vb(f)$, and for each index $j$ let $\lines_j$ be the set of lines from $\lines$ that intersect $\Omega_j$. 

For each index $j$ with $r> 2n_j^{1/2}$, define $\surfs_j=\emptyset$. Applying Lemma \ref{lem:rRichPtsBigR}, and recalling the assumption $r \le 2n^{1/2}$, we obtain the estimate
\begin{equation} \label{eq:LargeROne}
 |\pts_r(\lines') \cap \Omega_j| \le |\pts_r(\lines'_j)|\le 2n_jr^{-1} < 2nr^{-1} \le 4n^{3/2}r^{-2}.
 \end{equation}
Recall that $D$ depends only on $\eps$. If $n_0$ is selected sufficiently large compared to $\eps$ then 
\begin{equation} \label{eq:LargeRTwo}
4n^{3/2}r^{-2}\leq Cn^{3/2+\eps}r^{-2}D^{-6-4\eps}.
\end{equation}

For each index $j$ with $r\leq  2n_j^{1/2}$, apply the induction hypothesis to $\lines_j$ with the same values for $\eps$ and $r$. We obtain a set $\surfs_j$ of complex planes, such that 
\[ |\surfs_j|\le 2n_j^{1/2-\eps} = O(D^{-2+4\eps}n^{1/2-\eps}).\]

Define
$$
\surfs' = \bigcup_{j=1}^{s}\surfs_j.
$$
Note that $|\surfs'| = O_{D}\left(n^{1/2-\eps}\right)$.
Since $\pts_r(\lines')\cap \Omega_j \subset \pts_r(\lines'_j)$, we have
\begin{equation} \label{eq:SmallR}
|(\pts_r(\lines)\cap \Omega_j)\setminus \cup_{S\in \surfs_j}\pts_{r^\prime}(\lines_S)|\le C n_j^{3/2+\eps}r^{-2} = O(Cn^{3/2+\eps}D^{-6-4\eps}r^{-2}).
\end{equation}

Combining \eqref{eq:LargeROne}, \eqref{eq:LargeRTwo}, and  \eqref{eq:SmallR} and taking $D$ to be sufficiently large compared to $\eps$ gives
\begin{equation}
\begin{split}\label{eq:CellsBound}
\Big|(\pts_r(\lines)\backslash U)\setminus \bigcup_{S\in \surfs^\prime}\pts_{r^\prime}(\lines_S)\Big| &= O(D^6 \cdot Cn^{3/2+\eps}D^{-6-4\eps}r^{-2}) \\
&= O(kn^{3/2+\eps}D^{-4\eps}r^{-2})\\
& \le \frac{C}{4} n^{3/2+\eps} r^{-2}.
\end{split}
\end{equation}

Let $\surfs^{\prime\prime}$ be the set of complex planes $S\in\surfs^\prime$ that contain at least $2n^{1/2+\eps}r$ lines from $\lines$. We repeat the last part of the proof of Lemma \ref{lem:richLinesCurveIntersect}, which involved Theorem \ref{th:complexST} and ruled surfaces.
By the same argument
\begin{equation}\label{eq:throwAwayPoorSurfs}
\begin{split}
\sum_{S\in\surfs^{\prime}\backslash\surfs^{\prime\prime}}|\pts_{r^\prime}(\lines_S)|&\leq \frac{C}{4} n^{3/2+\eps} r^{-2}.
\end{split}
\end{equation}

It remains to derive an upper bound on the size of $\pts_r(\lines) \cap U$. Let $U_1 = U$, and for each $j=2,\ldots,6$ let $U_j = (U_{j-1})_{\sing}.$ By Lemma \ref{le:singular} each of the sets $U_1,\ldots,U_6$ are defined by polynomials of degree $O_{D}(1)$; $U_6$ is finite (possibly empty), and $U =\bigcup_{j=1}^6 (U_j)_{\reg}$. In particular, if $p\in\pts_r(\lines)\cap U$ then there is an index $j$ such that $p\in (U_j)_{\reg}$. Such a point $p$ is either incident to at least $r/2$ lines $L\in\lines$ that are contained in $U_j$, or to at least $r/2$ lines $L\in\lines$ that are not contained in $U_j$ (or both). 

For each index $j=1,\ldots,6$, we apply Lemma \ref{lem:incidencesInsideU} and Lemma \ref{lem:richLinesCurveIntersect} to $U_j$; we obtain sets $\surfs_j,\surfs_j^\prime$ of complex planes contained in $U_j$, with $|\lines_S|\geq 2n^{1/2+\eps}r$ for each $S\in\surfs_j\cup\surfs_j^\prime$. For each index $j$ we have
\begin{equation}\label{eq:linesInsideUi}
\begin{split}
\Big|(U_j)_{\reg}\cap\pts_{r/2}(\lines_{U_j})\backslash \bigcup_{S\in\surfs_j}\pts_{r/2}(\mathcal{L}_{S})\Big|= O_D( n^{3/2+\eps}r^{-2}+nr^{-1})\leq \frac{C}{4}n^{3/2+\eps}r^{-2},
\end{split}
\end{equation}
and 
\begin{equation}\label{eq:linesOutsideUi}
\Big|(U_j)_{\reg}\cap\pts_{r/2}(\lines\backslash \lines_{U_j})\backslash \bigcup_{S\in\surfs_j^\prime}\pts_{\frac{4}{5}\cdot\frac{r}{2}}(\mathcal{L}_{S})\Big|= O_D( n^{3/2+\eps}r^{-2}+nr^{-1})\leq \frac{C}{4}n^{3/2+\eps}r^{-2}.
\end{equation}

Let 
$$
\surfs = \surfs^{\prime\prime}\cup\bigcup_{j=1}^6 (\surfs_j\cup\surfs_j^\prime).
$$
Each plane $S\in\surfs$ contains at least $2n^{1/2+\eps}r$ lines from $\lines$, so by Lemma \ref{lem:FewRichSurfaces} we have $|\surfs|\leq n^{1/2-\eps}r^{-1}$. Combining \eqref{eq:CellsBound}, \eqref{eq:throwAwayPoorSurfs}, \eqref{eq:linesInsideUi}, and \eqref{eq:linesOutsideUi}, we obtain
\begin{equation*}
\Big|\pts_r(\lines)\setminus \bigcup_{S\in \surfs}\pts_{r/3}(\lines_S)\Big| \leq C n^{3/2+\eps} r^{-2}.
\end{equation*}
This closes the induction and completes the proof.
\end{proof}

Theorem \ref{th:RichC} can be used to obtain the following analogue of Theorem \ref{th:numberOfRRichPoints}
\begin{corollary}\label{cor:complexRRichPts}
Let $\lines$ be a set of at most $n$ complex lines in $\CC^3$. Suppose that at most $n^{1/2}$ lines from $\lines$ can be contained in a common plane or degree-two surface. Then for each $\eps>0$ there is a constant $C_{\eps}$ so that for all $2\leq r\leq 2n^{1/2}$,
\begin{equation*}
|\pts_r(\lines)|\leq C_{\eps} n^{3/2+\eps}r^{-2}
\end{equation*}
\end{corollary}

\section{The distinct distances problem} \label{sec:DDinC3}

We now study distinct distances in $\CC^2$. In \cite{GK15}, Guth and Katz used the Elekes-Sharir-Guth-Katz framework to convert an upper bound for incidences of lines in $\RR^3$ into a lower bound for distinct distances in $\RR^2$. In \cite{RR15}, Roche-Newton and Rudnev used a similar strategy to obtain a lower bound for the number of distinct ``Minkowski distances'' spanned by a set of points in $\RR^2$. If $p$ and $q$ are points in $\RR^2$, then the square of their Minkowski distance is the signed area of the rectangle with oppose corners $p$ and $q$. In contrast to the situation with Euclidean distances, it is possible for a pair of distinct points to have Minkowski distance zero. Roche-Newton and Rudnev introduced new arguments to tackle this situation. We will use similar ideas in order to use the incidence bound from Theorem \ref{th:RichC} to prove Theorem \ref{th:DDinC2}.

\subsection{The ESGK framework: from distinct distances to line intersections}
The first step in the ESGK framework is to reduce the problem of counting distinct distances to that of counting quadruples $a,b,c,d\in\CC^4$ with $\Delta(a,b)=\Delta(c,d)$. If $\pts\subset\CC^2$ is a finite set of points, we define
\[ Q(\pts) = \left\{(a,b,c,d)\in \pts^4 :\, \Delta(a,b)=\Delta(c,d)\neq 0\ \text{ and }\ (a,b)\neq (c,d) \right\}. \]
\begin{lemma}\label{lem:fromPtsToQuadruples}
Let $\pts$ be a set of $n$ points in $\CC^2$, at most third contained in a common isotropic line. Then
\begin{equation*}
|\Delta(\pts)|= \Omega(n^4 |Q(\pts)|^{-1}).
\end{equation*}
\end{lemma}
\begin{proof}
Recall that two points $p,q\in \CC^2$ satisfy $\Delta(p,q)=0$ if and only if there exists an isotropic line that contains both.
Indeed, if $p=(p_x,p_y)$ and $q=(q_x,q_y)$, then $\Delta(p,q)=(p_x-q_x)^2 + (p_y-q_y)^2$, so $\Delta(p,q)=0$ if and only if $(p_x-q_x)^2=-(p_y-q_y)^2$. This can occur if and only if $(p_x-q_x) = \pm i(p_y-q_y)$. 

Let $\Delta(\pts)\backslash\{0\}= \{\delta_1,\ldots,\delta_t\}$.
For every $1\le j \le t$, we set $N_j = |\{(a,b)\in \pts^2 :\, \Delta(a,b)=\delta_t\}|$.
Since every isotropic line contains at most $n/3$ points of $\pts$, every point $p\in \pts$ determines a nonzero distance with at least $n/3$ points of $\pts$. 
Since every ordered pair $(a,b)\in \pts^2$ with $\Delta(a,b)\neq 0$ contributes to exactly one $N_j$, we get that $\sum_{j=1}^t N_j = \Theta(n^2)$.

By Cauchy--Schwarz, we have
\begin{equation*}
|Q| = 2 \sum_{j=1}^t \binom{N_j}{2} \ge \frac{1}{2}\sum_{j=1}^t N_j^2 \ge \frac{\left(\sum_{j=1}^t N_j\right)^2}{2t} = \Omega\left(\frac{n^4}{|\Delta(\pts)|}\right).
\end{equation*}
\end{proof}

The second step in the ESKG framework is to reduce the problem of counting quadruples in $\mathcal{Q}(\pts)$ to that of counting line-line intersections in $\CC^3$. Given two distinct points $a=(a_x,a_y)$ and $c=(c_x,c_y)$ in $\CC^2$, we denote by $\ell_{a,c}$ the line in $\CC^3$ that is defined by the equations
\begin{align}
2x &= (a_x+c_x) +(a_y-c_y)z, \nonumber \\
2y &= (a_y+c_y) + (c_x-a_x)z. \label{eq:ComplexLineDef}
\end{align}

\begin{lemma}\label{lem:fromQuadruplesToLineIntersection}
Let $a,b,c,d\in\CC^2$. Then $\Delta(a,b)=\Delta(c,d)$ if and only if the lines $\ell_{a,c}$ and $\ell_{b,d}$ are coplanar.
\end{lemma}
Lemma \ref{lem:fromQuadruplesToLineIntersection} is proved in \cite[Lemma 4.2]{Guth15a} for the case of points in $\RR^2$. An identical proof works for points in $\CC^2$.

We say that a plane $\Pi\subset\CC^3$ is \emph{bad} if $\Pi = \vb(y\pm ix + k)$ for some $k\in \CC$. The next lemma shows that there is a correspondence between bad planes in $\CC^3$ and isotropic lines in $\CC^2$.

\begin{lemma} \label{cl:IsotropicChar}
Let $\ell^*\subset \CC^2$ be an isotropic line defined by $y=\pm ix + k$, where $k\in \CC$.
Let $\Pi\subset \CC^3$ be the plane defined by $y=\pm ix +k$ (with both $\pm$ signs representing the same symbol).
Then $\ell_{a,c} \subset \Pi$ if and only if $a,c\in \ell^*$.
\end{lemma}
\begin{proof}
We replace $\pm$ with a plus sign. 
The case of a minus sign is handled symmetrically. 
We first assume that $a,c\in \ell^*$ and prove that $\ell_{a,c} \subset \Pi$.
Write $a = (a_x,ia_x+k)$ and $c=(c_x,ic_x+k)$ for $a_x,c_x\in \CC$.
By \eqref{eq:ComplexLineDef}, the line $\ell_{a,c}$ is defined by
\begin{align*}
2x = (a_x+c_x) +iz(a_x-c_x) \quad \text{ and }\quad 2y = i(a_x+c_x)+2k + (c_x-a_x)z. 
\end{align*}

Combining these two equations leads to
\[ 2x+2iy = (a_x+c_x) +iz(a_x-c_x) + i\left(i(a_x+c_x) +2k + (c_x-a_x)z\right) = 2ik\]
Tidying up gives $iy=ik-x$ and multiplying by $-i$ gives $y=ix+k$. 
Thus, the line $\ell_{a,c}$ is contained in $\Pi$.

Next, suppose that $\ell_{a,c} \subset \Pi$.
If $a=c$ then $\ell_{a,c}$ is defined by $x=a_x$ and $y=a_y$.
Since $\Pi$ is defined by $y = ix+k$, we get that $a_y=ia_x+k$.
Thus, $a,c\in \ell^*$. 
It remains to consider the case where $a\neq c$. 

The line $\ell_{a,c}$ has a point satisfying $z=0$.
Plugging this point into 
\eqref{eq:ComplexLineDef}, we have 
\[ 2x = a_x+c_x \quad \text{ and } \quad 2y = a_y+c_y. \]
Since this point must also satisfy $y = ix+k$, we get
\begin{equation} \label{eq:IsotropicClaim} 
a_y+c_y = i(a_x+c_x)+2k.
\end{equation}

Assume for contradiction that $a_x = c_x$. 
Since $a\neq c$, we have $a_y\neq b_y$. 
By inspecting \eqref{eq:ComplexLineDef}, we note that the line $\ell_{a,c}$ has a constant $y$-coordinate and a non-constant $x$-coordinate. 
Since $\Pi$ does not contain any such lines, we obtain a contradiction.
We may thus assume that $a_x\neq c_x$.
A symmetric argument implies $a_y\neq c_y$.

Since $a_x\neq c_x$ and $a_y\neq c_y$, we may rewrite \eqref{eq:ComplexLineDef} as
\begin{align*}
z &= (2x-a_x-c_x)/(a_y-c_y), \\
z &= (2y-a_y-c_y)/(c_x-a_x). 
\end{align*}

Combining these equations leads to
\begin{equation} \label{eq:pointIsotroLineSlope} 
(2x-a_x-c_x)(c_x-a_x) = (2y-a_y-c_y)(a_y-c_y), \quad \text{ or } \quad y = x\cdot \frac{c_x-a_x}{a_y-c_y} + E, 
\end{equation}
where $E$ depends on $a_x,a_y,c_x,c_y$.
Since $\Pi$ is defined by $y = ix+k$ and $\ell_{a,c} \subset \Pi$, we obtain that 
\[ \frac{c_x-a_x}{a_y-c_y} = i \quad \text{ or equivalently } \quad a_x + ia_y = c_x+ic_y. \]

Combining this with \eqref{eq:IsotropicClaim} leads to $a_y=ia_x+k$ and $c_y=ic_x+k$.
That is, $a,c\in \ell^*$.
\end{proof}

If $\pts\subset\CC^2$ is a finite set of points, we define
\[\lines(\pts) = \{\ell_{a,c} :\, (a,c)\in \pts^2 \ \}.\]
\begin{lemma} \label{le:LinesRestrictionsComplex}
Let $\pts$ be a set of $n$ points in $\CC^2$. 
Then every point in $\CC^3$ is incident to at most $n$ lines from $\lines(\pts)$ and every irreducible surface of degree 2 contains at most $6n$ lines from $\lines(\pts)$. 
When no point of $\pts$ is on the lines $y=\pm ix$, every non-bad plane contains at most $2n$ lines from $\lines(\pts)$. 
\end{lemma}
\begin{proof}
Let $a,c,c' \in \CC^2$ be three distinct points.
By Lemma \ref{lem:fromQuadruplesToLineIntersection}, the lines $\ell_{a,c}$ and $\ell_{a,c'}$ are coplanar if and only if $\Delta(c,c')=0$.
That is, if and only if $c$ and $c'$ are on the same isotropic line. 

Fix a point $a\in \CC^2$ and an isotropic line $\ell^*$. We now show that all lines $\ell_{a,c}$ with $c\in \ell^*$ have a common intersection point. 
Assume that $\ell^*$ is defined by $y=ix+ k$ for some $k\in \CC$. 
By writing $c=(c_x,ic_x+k)$ and inspecting \eqref{eq:ComplexLineDef}, we note that $\ell_{a,c}$ is incident to the point 
\[ p=(a_x-ia_y+ik,ia_x+a_y+k,-i). \]

The coordinates of $p$ do not depend on $c$ and uniquely determine $a_x,a_y$, and $k$. 
For example, $a_x$ is the real part of the $x$-coordinate of $p$.
Thus, all lines $\ell_{a,c}$ with fixed $a$ and with $c$ on an isotropic line $\ell^*$ intersect at the same point.
The common intersection point $p$ has a $z$-coordinate equal to $-i$. 
(If $\ell^*$ is instead defined as $y=-ix+ k$, then the $z$-coordinate becomes $i$.)
A line $\ell_{b,d}$ with $a\neq b$ or with $d$ not on $\ell^*$ intersects $z=\pm i$ at a different point.

Let $p=(p_x,p_y,p_z)\in \CC^3$ satisfy $p_z\neq \pm i$. 
By the above, for every $a\in \pts$ there exists at most one $c\in \CC^3$ such that $p\in \ell_{a,c}$.
Thus, at most $n$ lines from $\lines(\pts)$ are incident to $p$.

Let $p=(p_x,p_y,p_z)\in \CC^3$ satisfy $p_z= \pm i$. 
By the above, there exists a unique $a\in \CC^3$ such that some lines of the form $\ell_{a,c}$ are incident to $p$.
Once again, at most $n$ lines from $\lines(\pts)$ are incident to $p$.

\parag{Planes.}
Let $\Pi\subset \CC^3$ be a plane that is not bad.
Fix $a\in \pts$ and consider the set $\lines_a = \{\ell_{ac}\ :\ c\in \pts \}$.
By inspecting \eqref{eq:ComplexLineDef}, we note that no two lines in $\lines_a$ are parallel.
Thus, every pair of lines of $\lines_a$ that are in $\Pi$ intersect. 
By the above, for lines $\ell_{a,c}$ and $\ell_{a,c'}$ to intersect, the points $c$ and $c'$ must lie on a common exceptional line.
That is, there exists an exceptional line $\ell^*$ such that every $\ell_{a,c}\subset \Pi$ satisfies $c\in \ell^*$.

Let $a,\Pi$, and $\ell^*$ be fixed as in the preceding paragraph. 
Let $\ell^*$ be defined by $y=ix+k$ (the case of $y=-ix+k$ is handled symmetrically).
Then for every $c\in \ell^*$, the line $\ell_{a,c}$ is defined as
\begin{align*}
2x &= (a_x+c_x) +(a_y-ic_x-k)z, \nonumber \\
2y &= (a_y+ic_x+k) + (c_x-a_x)z. 
\end{align*}
We rewrite these equations as
\begin{align*}
c_x(1-iz) &= 2x -a_x -a_yz+kz, \nonumber \\
c_x(i+z) &= 2y - a_y-k + a_xz. 
\end{align*}

Combining the above equations leads to
\[ (i+z)(2x -a_x -a_yz+kz) = (1-iz)(2y - a_y-k + a_xz).\]
Rearranging yields
\begin{equation}\label{eq:QuadraticLinesSpecialFamily}
z^2\cdot (k-a_y+ia_x)+2z\cdot(x-ia_y-a_x+iy)+ (2ix-a_xi-2y+a_y+k)=0.
\end{equation}

Denote the left side of \eqref{eq:QuadraticLinesSpecialFamily} as $f_a\in\CC[x,y,z]$.
Note that $\vb(f_a)$ is the Zariski closure of the union of the lines $\ell_{a,c}$ with $c\in \ell^*$.
If $f$ is irreducible, then Corollary \ref{co:BezoutLines} implies that $\Pi\cap \vb(f_a)$ contains at most two lines. 
That is, $\Pi$ contains at most two lines of the form $\ell_{a,c}$ with $c\in \ell^*$.
We conclude that, if $f_a$ is irreducible then $\Pi$ contains at most two lines from $\lines_a$.

Consider the case where the coefficient of $z^2$ in $f_a$ is zero. 
That is, $k-a_y+ia_x=0$.
In this case, $a$ is also on the exceptional line $\ell^*$. 
Lemma \ref{cl:IsotropicChar} implies that all lines of the form $\ell_{a,c}$ with $c\in \ell^*$ lie on a specific bad plane $\Pi^*$.
Since $\Pi$ is not bad, $\Pi \cap \Pi^*$ contains at most one line.
Thus, in this case $\Pi$ contains at most one line from $\lines_a$.

Finally, assume that $f_a$ is reducible and has a nonzero coefficient for $z^2$. 
In this case $k-a_y+ia_x\neq0$, which implies that $f_a$ has a nonzero constant coefficient and a nonzero coefficient for $z$.
By inspecting \eqref{eq:QuadraticLinesSpecialFamily}, we note that $f_a(x,y,z)=(A+Bx+Cy+Dz)(E+Fz)$ for nonzero $A,B,C,D,E,F\in \CC$. 
Since $EB$ is the coefficient of $x$ in $f_a$, we have that $EB=2i$.
Similarly, we obtain
\begin{align*} 
EC &= -2, \quad FB=2, \quad FC=2i,\quad FD = k-a_y+ia_x, \\[2mm]
EA&= -a_xi+a_y+k,\quad FA+ED = -ia_y-a_x.
\end{align*}

We rewrite some of the above as
\begin{align*} 
C&=-2/E, \quad F=2i/C=-iE, \quad A=(-a_xi+a_y+k)/E, \\[2mm]
D&=(k-a_y+ia_x)/F = (ki-ia_y-a_x)/E.  
\end{align*}
Combining this with the above expression for $FA+ED$ leads to
\[ -ia_y-a_x = FA+ED = (-iE) \cdot \frac{-a_xi+a_y+k}{E} + E\cdot \frac{ki-ia_y-a_x}{E} = -2a_x-2ia_y. \]

Tidying up gives that $a_y=ia_x$.
Plugging this into \eqref{eq:QuadraticLinesSpecialFamily}, we get
\[ f_a(x,y,z)=z^2k+2z\cdot(x+iy)+ (2ix-2y+k)=(k+2ix-2y+kiz)(1-iz). \]

By combining the three cases above, we obtain the following. 
When no point of $\pts$ is on the lines $y=\pm ix$, for every $a\in \pts$ the plane $\Pi$ contains at most two lines from $\lines_a$.
Thus, $\Pi$ contains at most $2n$ lines of $\lines(\pts)$. 

\parag{Quadratic surfaces.}
For a point $a\in \CC^3$, consider the vector field
\begin{align*} 
V_a(x,y,z) &= \Bigg( 2y-a_y-a_xz + z(2x-a_x+a_yz)-a_y(z^2+1), \\
&\hspace{11mm} a_x(z^2+1)-2x -a_x+a_yz - z(2y-a_y-a_xz),2(z^2+1)\Bigg).
\end{align*}
By repeating a proof from \cite{GK15}, we obtain the following property. 
For every $p=(p_x,p_y,p_z)\in \CC^3$ with $p_z\neq \pm i$, the direction of the unique line of the form $\ell_{a,c}$ that is incident to $p$ is $V_a(p)$.

Let $U\subset \CC^3$ be an irreducible quadratic surface.
Let $f\in \CC[x,y,z]$ be a polynomial of degree 2 satisfying $\vb(f)=U$.
We define $g_a\in \CC[x,y,z]$ as the dot product $g_a(p) = V_a(p) \cdot \nabla f(p)$.
Consider a line $\ell_{a,c}$ that is contained in $U$.
Then at every point $p\in \ell_{a,c}$, we have that $V_a(p)$ is tangent to $U$. 
Thus, $g_a$ vanishes on every line of the form $\ell_{a,c}$ that is contained in $U$.

Consider $a\in \CC^3$ such that $U$ contains at least five lines of the form $\ell_{a,c}$.
Then these lines are contained in $U\cap \vb(g_a)$.
Since $f$ is of degree 2 and $V_a$ is linear, we get that $g_a$ is of degree at most two. 
Thus, Corollary \ref{co:BezoutLines} implies that $U$ and $\vb(g_a)$ have a common component. 
Since $U$ is irreducible, we have $U=\vb(g_a)$.
This in turn implies that $U$ is ruled by lines of the form $\ell_{a,c}$.

Excluding planes, every irreducible surface in $\CC^3$ has at most two different rulings. 
Thus, there are at most two points $a\in\pts$ such that lines of the form $\ell_{a,c}$ rule $U$.
For every such $a$, at most $n$ lines of the form $\ell_{a,c}$ are in $\lines(\pts)$.
For every other value of $a\in \pts$, at most four lines of the form $\ell_{a,c}$ are contained in $U$.
We conclude that $U$ contains fewer than $6n$ lines from $\lines(\pts)$ 
\end{proof}

\begin{lemma}\label{lem:FromQToIntersections}
Let $\pts$ be a finite set of points in $\CC^2$, and let $\lines=\lines(\pts)$. 
\begin{equation*}
|\mathcal{Q}(\pts)| \le |\{(L,L^\prime)\in\lines^2\colon L\ \textrm{and}\ L^\prime\ \textrm{are contained in a common non-bad plane}\}|.
\end{equation*}
\end{lemma}
\begin{proof}
Consider a quadruple $(a,b,c,d)\in \pts^4$ with $(a,b)\neq (c,d)$. 
To prove the lemma, we show that for every $(a,b,c,d)\in \mathcal{Q}(\pts)$, the lines $\ell_{a,c}$ and $\ell_{b,d}$ are contained in a common non-bad plane.

Recall that $(a,b,c,d)\in \mathcal{Q}(\pts)$ if and only if $\Delta(a,b)=\Delta(c,d)\neq 0$. 
By Lemma \ref{lem:fromQuadruplesToLineIntersection}, the lines $\ell_{a,c}$ and $\ell_{b,d}$ are coplanar if and only if $\Delta(a,b)=\Delta(c,d)$.
It remains to show that, if $\Delta(a,b)=\Delta(c,d)\neq 0$ then $\ell_{a,c}$ and $\ell_{b,d}$ are not contained in the same bad plane.

Assume that $\ell_{a,c}$ and $\ell_{b,d}$ are contained in the same bad plane. Then Lemma \ref{cl:IsotropicChar} implies that $a,b,c,d$ are on a common isotropic line.
This in turn implies the contradiction $\Delta(a,b)=\Delta(c,d)=0$.
This is the contrapositive of what we need to show, so it concludes the proof. 
\end{proof}

\begin{lemma}\label{lem:removeParallelLines}
Let $\pts$ be a set of $n$ points in $\CC^2$. There are $O(n^3)$ pairs of parallel lines in $\lines(\pts)$.
\end{lemma}
\begin{proof}
If $\ell_{a,c}$ and $\ell_{b,d}$ are parallel, then \eqref{eq:ComplexLineDef} implies that $a_y-c_y = b_y-d_y$ and $a_x-c_x = b_x-d_x$.
If we fix $a,b,c\in \pts$, there is at most one $d\in \pts$ that satisfies these two equations.
Thus, at most $n^3$ pairs of $(\lines(\pts))^2$ are parallel.
\end{proof}

\subsection{Controlling line-line intersections}
To summarize our progress so far, Lemma \ref{lem:fromPtsToQuadruples} reduced the problem of lower-bounding $|\Delta(\pts)|$ to the problem of upper-bounding $|Q(\pts)|$. Lemmas \ref{lem:FromQToIntersections} and \ref{lem:removeParallelLines} reduced the problem of upper-bounding $|Q(\pts)|$ to the problem of bounding the number of pairs of lines from $\lines(\pts)$ that intersect and do not span a bad plane. In this section we use Theorem \ref{th:RichC} to bound this quantity.

\begin{lemma}\label{lem:numberOfGoodQuadruples}
Let $\lines$ be a set of at most $n$ complex lines in $\CC^3$. Suppose that at most $n^{1/2}$ lines from $\lines$ are contained in a common irreducible degree-two surface and at most $n^{1/2}$ lines pass through a common point. Let $\surfs$ be the set of complex planes that contain at least $2n^{1/2}$ lines from $\lines$. Then for each $\eps>0$, there is a constant $C_\eps$ such that
\begin{equation}\label{eq:numberOfLLpPairs}
|\{(L, L^\prime)\in\lines^2\colon L\cap L^\prime\neq\emptyset;\ L\ \textrm{and}\ L^\prime\ \textrm{do not span a plane from}\ \surfs\}| \leq C_{\eps}n^{3/2+\eps}.
\end{equation}
\end{lemma}
\begin{proof}
For each $r\geq 2$, define 
$$
\pts_{\sim r}(\lines)=\pts_{r}(\lines)\backslash\pts_{2r}(\lines).
$$ 
For each dyadic $r$ between $2$ and $n^{1/2}$, we apply Theorem \ref{th:RichC} to $\lines$. For $r^\prime = \max(2,  r/3)$, we obtain a set $\surfs_r \subset\surfs$ such that $|\surfs_r|\leq n^{1/2}r^{-1}$ and 
\begin{equation}\label{eq:consequenceOfThNumberOfRRichPoints}
|\pts_{\sim r}(\lines)\backslash \bigcup_{S\in\surfs}(\pts_{r^\prime}(\lines_S))| = O( n^{3/2+\eps/2}r^{-2}).
\end{equation}

We define 
\begin{align*}
\pts_{\sim r}^{(1)}(\lines)&=\pts_{\sim r}(\lines)\backslash \bigcup_{S\in\surfs}\pts_{r^\prime}(\lines_S), \\[2mm]
\pts_{\sim r}^{(2)}(\lines)&=\pts_{\sim r}(\lines)\backslash \pts_{\sim r}^{(1)}(\lines).
\end{align*}
By \eqref{eq:consequenceOfThNumberOfRRichPoints}, we have
\begin{equation*}
|\{(L,L^\prime)\in \lines^2\colon L\cap L^\prime= p\in \pts_{\sim r}^{(1)}(\lines)\}|=O(n^{3/2+\eps/2}). 
\end{equation*}

Fix $p\in \pts_{\sim r}^{(2)}(\lines)$.
Then there is a plane $S_p\in\surfs$ with $p\in\pts_{r^\prime}(\lines_{S_p})$. Since $r^\prime=\Theta(r)$, we have
\begin{align}
|\{(L,L^\prime)\in \lines^2 & \colon L\cap L^\prime= p,\ L\ \textrm{and}\ L^\prime\ \textrm{do not span a plane from}\ \surfs_r\}|\nonumber \\
& \hspace{20mm} = O(|\{(L,L^\prime)\in \lines^2 \colon L\cap L^\prime= p,\ L\not\subset S_p,\ L^\prime\subset S_p\}|).\label{eq:crossTermsDominated}
\end{align}

Fix $L\in \lines$. 
An important observation is that if $L$ is not contained in a plane $S\in\surfs_r$, then $L\cap S$ is a point.
That is, at most one point $p\in L\cap \pts_{\sim r}^{(2)}(\lines)$ satisfies $S_p=S$. By definition, there are fewer than $2r$ lines $L^\prime\in\lines_S$ with $L\cap L^\prime=p$. In other words, 
\begin{equation}\label{eq:fewLinesInPlaneHitL}
\sum_{\substack{p\in \pts_{\sim r}^{(2)}(\lines)\\ S_p = S}}|\{L^\prime\in \lines_{S} \colon  L\cap L^\prime= p\}|\leq 2r.
\end{equation}
(Since the sum has at most one nonzero element.)
Combining \eqref{eq:crossTermsDominated} and \eqref{eq:fewLinesInPlaneHitL}, we have
\begin{equation*}
\begin{split}
|\{(L,L^\prime)\in \lines^2& \colon L\cap L^\prime\in \pts_{\sim r}^{(2)}(\lines),\ \ L\ \textrm{and}\ L^\prime\ \textrm{do not span a plane from}\ \surfs_r\}|\\
& =O(|\{(L,L^\prime)\in \lines^2 \colon L\cap L^\prime= p\in \pts_{\sim r}^{(2)}(\lines),\ L\not\subset S_p,\ L^\prime\subset S_p\}|)\\
&=O( \sum_{S\in \surfs_r}\ \ \sum_{\substack{L\in\lines\\ L\not\subset S}}\ \ \sum_{\substack{p\in \pts_{\sim r}^{(2)}(\lines)\\ S_p = S}}|\{L^\prime\in \lines_{S} \colon  L\cap L^\prime= p\}|)\\
&=O( \sum_{S\in \surfs_r}\ \ \sum_{\substack{L\in\lines\\ L\not\subset S}}2r)\\
&=O( n^{1/2-\eps}r^{-1} \cdot n\cdot 2r) =O(n^{3/2}).
\end{split}
\end{equation*}

Finally, let $k_0$ be the smallest integer with $2^{k_0}\geq n^{1/2}$. We have
\begin{equation*}
\begin{split}
|\{(L, L^\prime)& \in\lines^2\colon L\cap L^\prime\neq\emptyset;\ L\ \textrm{and}\ L^\prime\ \textrm{do not span a plane from}\ \surfs\}|\\
&\leq \sum_{k=1}^{k_0} |\{(L,L^\prime)\in \lines^2\colon L\cap L^\prime\in \pts_{\sim 2^k}^{(1)}(\lines)\}| \\
&\ \ \ + \sum_{k=1}^{k_0} |\{(L,L^\prime)\in \lines^2 \colon L\cap L^\prime\in \pts_{\sim r}^{(2)}(\lines), \ L, L^\prime\ \textrm{do not span plane from}\ \surfs_r\}|\\
&=O(( n^{3/2+\eps/2} + n^{3/2})\log n).
\end{split}
\end{equation*}
By selecting a sufficiently large $C_{\eps}$, we obtain \eqref{eq:numberOfLLpPairs}.
\end{proof}

\subsection{Proof of Theorem \ref{th:DDinC2}}
We are now ready to prove Theorem \ref{th:DDinC2}. We first recall the statement of this theorem.

\begin{thmDDinC2}
For every $\eps>0$, there is a positive constant $c$ such that the following holds. Let $\pts$ be a set of $n$ points in $\CC^2$, not all on the same isotropic line. Then 
\begin{equation*}
|\Delta(\mathcal{P})|\geq cn^{1-\eps}.
\end{equation*}
\end{thmDDinC2}
\begin{proof}
Fix $\eps>0$. Suppose that there exists an isotropic line $\ell \subset \CC^2$ that contains at least $n/2$ points of $\pts$.
By assumption, there exists at least one point $p\in \pts$ that is not on $\ell$.
For any nonzero distance $\delta\in\Delta(\pts)$, at most two points $q\in \ell$ satisfy $\Delta(p,q)=\delta$. Indeed, any such point must be contained in the intersection of $\ell$ with the circle $\{(x,y)\in\CC^2\colon (x-p_x)^2 + (y-p_y)^2 = \delta\}$.
No point on $\ell$ determines distance 0 with $p$. 
This implies that the number of distinct distances determined by pairs of $\{p\} \times (\pts \cap \ell)$ is at least $|\pts \cap \ell|/2 = \Theta(n)$.
We may thus assume that every isotropic line contains at most $n/3$. By Lemma \ref{lem:fromPtsToQuadruples} we have
\begin{equation}\label{eq:deltaVsQ} 
|\Delta(\pts)|=\Omega( n^4|Q(\pts)|^{-1}).
\end{equation}

By Lemma \ref{lem:FromQToIntersections}, $|Q|$ is at most the number of pairs of distinct lines from $\lines(\pts)$ that are coplananr and do not span a bad plane. Write $Q = Q_1\cup Q_2$, where $Q_1$ corresponds to pairs of parallel lines, and $Q_2$ corresponds to pairs of lines that intersect. By Lemma \ref{lem:removeParallelLines} we have 
\begin{equation}\label{eq:BdOnQ1}
|Q_1| = O(n^3).
\end{equation} 
By Lemma \ref{le:LinesRestrictionsComplex}, each non-bad plane and each irreducible degree-two surface contains at most $6n$ lines from $\lines(\pts)$. In particular if we define $\surfs$ to be the set of surfaces contain at least $6n$ lines from $\lines(\pts)$, then $\surfs$ consists entirely of bad planes. Thus by Lemma \ref{lem:numberOfGoodQuadruples}, we have that
\begin{equation}\label{eq:BdonQ2}
|Q_2|\leq C_{\eps}n^{3+\eps}.
\end{equation}  
Combining \eqref{eq:deltaVsQ}, \eqref{eq:BdOnQ1} and \eqref{eq:BdonQ2} completes the proof.
\end{proof}

\end{document}